\documentclass[12pt]{amsart}
\usepackage[utf8]{inputenc}
\usepackage{amsbsy,amssymb,amsfonts}
\usepackage{mathtools}
\usepackage[top=3cm,left=3cm,right=3cm,bottom=3cm]{geometry}
\usepackage{graphicx}
\usepackage{xcolor}
\usepackage{enumerate}
\usepackage[bookmarks=true,hyperindex,pdftex,colorlinks,citecolor=red, linkcolor=blue]{hyperref}
\usepackage{float}

\DeclareMathOperator{\supp}{supp}
\DeclareMathOperator{\Chi}{Chi}
\DeclareMathOperator{\Gen}{Gen}
\DeclareMathOperator{\Par}{par}
\DeclareMathOperator{\Orb}{Orb}

\newtheorem{theorem}{Theorem}[section]
\newtheorem{lemma}[theorem]{Lemma}
\newtheorem{proposition}[theorem]{Proposition}
\newtheorem{corollary}[theorem]{Corollary}

{\theoremstyle{definition}
\newtheorem{example}[theorem]{Example}
\newtheorem{question}{Question}
}

\newcommand{\dis}{\displaystyle}
\newcommand{\veps}{\varepsilon}

\def\root{\texttt{r}}
\def\RR{\mathbb R}
\def\NN{\mathbb N}
\def\ZZ{\mathbb Z}
\def\CC{\mathbb C}
\def\KK{\mathbb K}
\def\A{\mathcal A}

\title{Hypercyclic algebras for weighted shifts on trees}
\author{Arafat Abbar}
\author{Fernando Costa Jr.}
\date{\today}

\address[A. Abbar]{Univ Gustave Eiffel, Univ Paris Est Creteil, CNRS, LAMA UMR8050 F-77447 Marne-la-Vallée, France}
\email{abbar.arafat@gmail.com}

\address[F. Costa Jr.]{Universidade Federal da Paraíba - Campus I, Departamento de Matemática, Jardim Universitário, s/n, Bairro Castelo Branco, CEP 58051-900, João Pessoa, Brazil}
\email{fernando@mat.ufpb.br}

 \allowdisplaybreaks

\begin{document}


\begin{abstract}
We study the existence of algebras of hypercyclic vectors for weighted backward shifts on sequence spaces of directed trees with the coordinatewise product. When $V$ is a rooted directed tree, we show the set of hypercyclic vectors of any backward weighted shift operator on the space $c_0(V)$ or $\ell^1(V)$ is algebrable whenever it is not empty. We provide necessary and sufficient conditions for the existence of these structures on $\ell^p(V), 1<p<+\infty$. Examples of hypercyclic operators not having a hypercyclic algebra are found. We also study the existence of mixing and non-mixing backward weighted shift operators on any rooted directed tree, with or without hypercyclic algebras. The case of unrooted trees is also studied.
\end{abstract}

\maketitle

\section{Introduction}
In the field of Linear Dynamics, one studies various dynamical properties satisfied by a linear map $T:X\to X$ acting on a linear space $X$. One of the main notion in the theory is that of \emph{hypercyclicity}, that is, the presence in $X$ of a dense orbit under the action of $T$. Let $X$ be a complex separable Banach space and $T:X\to X$ be a bounded linear operator. We say that $T$ is \emph{hypercyclic} when there is a vector $x\in X$ whose orbit $\Orb(x,T):=\{T^nx:\, n\in\NN_0\}$ is dense in $X$, where $\NN_0:=\NN\cup\{0\}$. Such a vector is called a \emph{hypercyclic vector} for the operator $T$, which in this case is said to be a \emph{hypercyclic operator}. The set of hypercyclic vectors for an operator $T$ will be denoted by $HC(T)$. Due to the importance of this topic, it is not a huge surprise that there is a big interest in the study of the structure of the set of hypercyclic vectors for a given hypercyclic operator. It is known from \cite[Theorem 9.20]{GottHedl} that $HC(T)$ is always either empty or residual, and that in some cases it can be spaceable (as shown by Bernal-Gonz\'alez and Montes-Rodr\'iguez in \cite{bernal1995}) and even algebrable (as shown by B\`es and Papathanasiou in \cite{bes2020algebrable}).

One of the most studied classes of operators in Linear Dynamics is that of backward shifts. This is of course not a surprise: shift operators are simple to define; they are highly flexible for examples and counter-examples; their orbits are easy to calculate; they can present the most wild behavior; in many cases, other classes of operators can be identified as shifts in some sense. Thus, it is natural to study the aforementioned properties in this context, what has been done on rather general Fréchet sequence spaces. In particular, when it comes to weighted backward shifts operators $B_w$ acting on spaces like $c_0(\NN)$ or $\ell^p(\NN), 1\leq p<+\infty$, (or even $c_0(\ZZ)$ and $\ell^p(\ZZ)$), their hypercyclicity and the spaceability of $HC(B_w)$ are completely characterized (see \cite{salas}, \cite[Section 4.1]{karlbook} and \cite{menet2014subspaces}). The search for algebras of hypercyclic vector is also of great interest in the literature (see for example \cite{bes2018algebra}, \cite{bayart2019algebra} and \cite{besernst2020algebra}). In particular, sufficient conditions for the algebrability of $HC(B_w)$ were obtained in \cite{karl-falco} and \cite{BCP}. These last results deal with ``classical'' sequences spaces. In the recent work \cite{Karl1} by K.-G. Grosse-Erdmann and D. Papathanasiou, among other things, they characterize hypercyclic and mixing weighted backward shifts acting on sequence spaces of a directed tree. Continuing this trend, in this paper we aim to study hypercyclic algebras in this context. More generally, we are interested in the algebrability of $HC(T)$ when $T$ is a weighted backward shift acting on the sequence space of a tree, which is a Banach algebra when endowed with the coordinatewise product.

In Section \ref{sec:prel}, we present the main definitions and key results that will be used throughout this work. Section \ref{sec:rooted} addresses the case of rooted directed trees: we characterize weighted backward shift operators on $\ell^p$-spaces, $1 \leq p < +\infty$, and $c_0$-spaces of directed trees that have hypercyclic algebras, and we provide some examples and counterexamples in the case of $\ell^p$-spaces with $1<p<+\infty$. In Section \ref{sec:unrooted}, we turn our attention to the unrooted case, offering sufficient conditions for weighted backward shifts to support hypercyclic algebras. Section \ref{sec:existence} 
then explores the existence of backward shifts that support a hypercyclic algebra on a given $\ell^p$-space, $1 \leq p < +\infty$, or $c_0$-space of a tree. Finally, in Section \ref{sec:conc}, we discuss some challenges, pose open problems, and provide perspectives for future research on this topic.

\section{Preliminaries}\label{sec:prel}
\subsection{Weighted shifts operators on trees}

A directed tree is a connected directed graph $(A,E)$, where $A$ is a countable set whose elements are called \textit{vertices}, and $E\subset\{(a,b)\in A\times A:\, a\neq b\}$  represents the set of \textit{edges}. This graph possesses the following properties: it has no cycles, meaning there are no paths that lead back to a vertex already visited, and each vertex $b \in A$ has at most one parent. The parent of a vertex $b \in A$, denoted by $\Par(b)$, is a vertex $a \in A$ for which $(a,b) \in E$. More generally, $\Par^k(b):=\Par\big(\Par^{k-1}(b)\big)$ for every integer $k\geq2$. There is at most one vertex without a parent, called \textit{root}, and denoted simply by $\root$. Therefore, a rooted tree is a tree that has a root. For more details, we refer to \cite{Karl1,Jablonski, serre2002trees} and the references therein.

The parent-child relationship can also be expressed as $b$ being a child of $a$ whenever $a$ is the parent of $b$. The set $\Chi(a)$ stands for the set of children of a vertex $a$. We do not make any restrictions on the number of children a vertex can have (notice that they are at most countable by definition). However, we consider only leafless trees, meaning every vertex must have at least one child. This assumption is necessary for the existence of hypercyclic shifts operators, see \cite[Remark 4.1]{Karl1}.

Weighted backward shift operators are defined on $\KK^A$ (the space of maps from $A$ into $\KK$, where $\KK$ is either the real field $\RR$ or the complex field $\CC$) by, for any $f\in\KK^A$ and $v\in A$,
 \[(B_{\lambda}f)(v)=\sum_{u\in\Chi(v)}\lambda_u f(u),\]
where $\lambda=(\lambda_u)_{u\in A}$ is a family of non-zero scalars called \textit{weight}.

For every integer $n\geq1$, the $n$-th iteration of a weighted backward shift on a directed tree can be computed as
\[(B_{\lambda}^{n}f)(v)=\sum_{u\in\Chi^n(v)}\lambda(v\to u)f(u),\]
where $\lambda(v\to u):=\prod_{k=0}^{n-1}\lambda_{\Par^{k}(u)}$, $\Chi^0(v):=\{v\}$ and 
\[\Chi^n(v):=\bigcup_{u\in \Chi^{n-1}(v)}\Chi(u).\]
For any $1\leq p<+\infty$, we denote by $\ell^p(A)$ the space of $p$-summable families indexed by $A$, that is, 
\[\ell^p(A)=\Big\{f\in\KK^A:\, \|f\|_p:=\Big(\sum_{v\in A}|f(v)|^p\Big)^{1/p}<\infty\Big\}.\]
We also consider the $c_0$-space of $A$, which is defined by
\[c_0(A):=\Big\{f\in\KK^A:\, \forall\varepsilon>0, \exists F\subset A \text{ finite}, \forall v\in A\setminus F , \, |f(v)|<\varepsilon \Big\}\]
and is equipped with the sup norm. The canonical basis $(e_v)_{v\in A}$ is defined, for each $v\in A$, by $e_v=\chi_{\{v\}}$. The support of a function $f\in \KK^A$ is defined by $\supp(f)=\{v\in A : f(v)\neq 0\}$. We say that \emph{$f\in \KK^A$ has finite support in $F$} when $\supp(f)\subset F$ and $F\subset A$ is finite. In this case, we can write $f=\sum_{v\in F}f(v)e_v$. Note that the space of functions in $\KK^A$ with finite support is dense in $\ell^p(A)$, $1\leq p<+\infty$, as well as in $c_0(A)$. 

In particular for $c_0$-spaces, the following statement from \cite[Lemma 4.2]{Karl1} will be quite useful for constructing right inverses for backward shifts. Let $J$ be a countable set and let $\mu=(\mu_j)_j\in (\mathbb K\backslash \{0\})^J$. Then
\begin{equation}\label{key-lemma}
    \inf_{\|x\|_1=1}\sup_{j\in J}|x_j\mu_j|=\bigg(\sum_{j\in J}\frac{1}{|\mu_j|}\bigg)^{-1}, \quad \text{ where } \infty^{-1}:=0.
\end{equation}
The same holds when the sequences $x$ are required, in addition, to be of finite support.

Weighted forward shifts on directed trees were introduced and studied by Jabloński, Jung, and Stochel in \cite{Jablonski}. Martínez-Avendaño initiated the study of hypercyclicity of forward and backward shifts on weighted $\ell^p$-spaces on directed trees \cite{martinez2017}. Complete characterizations for hypercyclicity, weak mixing, and mixing properties of weighted backward shifts on $c_0$-spaces and $\ell^p$-spaces on directed trees were obtained by  Grosse-Erdmann and Papathanasiou in \cite{Karl1}. They also investigated the chaoticity of these operators in \cite{Karl2}. In \cite{aba2024sev},   Abakumov and the first author characterized $\mathcal{F}$-transitivity and topological $\mathcal{F}$-recurrence of these operators, where $\mathcal{F}$ is a Furstenberg family of subsets of the set of non-negative integers. For other related works on this topic, see \cite{baranov, lopez, rivera2019}.

The following proposition summarizes the boundedness of weighted Backward shift operators on  $c_0$-spaces and $\ell^p$-spaces on  directed trees, as shown in \cite[Proposition 2.3]{Karl1}.

\begin{proposition}\label{boundedness}
    Let $A$ be a directed tree and $X=\ell^p(A), 1\leq p<\infty$, or $X=c_0(A)$. The following conditions characterize the weights $\lambda=(\lambda_v)_{v\in A}$ for which $B_\lambda$ is bounded.
    \begin{enumerate}
        \item For $X=\ell^1(A)$, the condition is $\|B_\lambda\|:=\sup_{v\in A}\, |\lambda_v|<\infty$.
        \item For $X=\ell^p(A), 1<p<\infty$, the condition is $\|B_\lambda\|:=\sup_{v\in A}\Bigl(\sum_{u\in\Chi(v)}|\lambda_u|^{p^*}\Bigr)^{\frac{1}{p^*}}\!<\infty$, where $p^*$ is the conjugate exponent of $p$.
        \item For $X=c_0(A)$, the condition is $\|B_\lambda\|:=\sup_{v\in A}\Bigl(\sum_{u\in\Chi(v)}|\lambda_u|\Bigr)<\infty$.
    \end{enumerate}
\end{proposition}

For the convenience of the reader, we will recall the following characterization of hypercyclicity for weighted Backward shift operators on rooted directed trees, obtained by K.-G. Grosse-Erdmann and D. Papathanasiou. 
\begin{theorem}[\protect{\cite[Theorem 4.4]{Karl1}}]
    Let $A$ be a rooted directed tree and $\lambda=(\lambda_v)_{v\in A}$ be a weight. Let $X = \ell^p(A), 1 \leq p <+\infty$, or $X = c_0(A)$, and suppose that the weighted backward shift $B_\lambda$ is a bounded operator on $X$. Then $B_\lambda$ is hypercyclic if and only if there is an increasing sequence of positive integers $(n_k)_{k\geq1}$ such that, for every $v\in A$,
        \begin{align*}
            \sup_{u\in \Chi^{n_k}(v)}\big|\lambda(v\to u)\big|\xrightarrow{k\to+\infty}+\infty & \text{ if } X=\ell^1(A);\\
            \sum_{u\in \Chi^{n_k}(v)}\big|\lambda(v\to u)\big|^{p^*} \xrightarrow{k\to+\infty}+\infty & \text{ if } X=\ell^p(A), 1<p<\infty;\\
            \sum_{u\in \Chi^{n_k}(v)}\big|\lambda(v\to u)\big|\xrightarrow{k\to+\infty}+\infty  & \text{ if } X=c_0(A).
        \end{align*}
    \label{carac-c0} \label{carac-l1} \label{carac-lp}
\end{theorem}
We conclude this subsection by introducing some notations that will be used throughout the paper. In a rooted directed tree with 
root 
$\texttt{r}$, the set
\[\mathrm{Gen}_n:=\Chi^n(\texttt{r}),\, n\in \NN_0,\]
is called the \emph{$n$-th generation}. In an unrooted directed tree, the \emph{$n$-th generation}, $n\in\ZZ$, with respect to a given vertex $v_0$, denoted as $\mathrm{Gen}_n:=\mathrm{Gen}_n(v_0)$,  consists of all vertices that belong to $\Chi^{k+n}(\Par^k(v_0))$ for some $k \geq  \max\{-n,0\}$. 

For two vertices $u$ and $v$ in a directed tree, we write $u \sim v$ if they belong to the same generation. A weight $\lambda=(\lambda_v)_{v\in A}$ for a directed tree $A$ is called \emph{symmetric} if $\lambda_u=\lambda_v$ whenever $u\sim v$.

Furthermore, an unrooted directed tree is said to have a \emph{free left end} if there exists some integer $n$ such that $\mathrm{Gen}_k$ is a singleton for all $k\leq n$. It is interesting to notice that an unrooted tree has a free left end if and only if it has at least one finite generation, in which case it has infinitely many. 

\subsection{Hypercyclic algebras}

As we have already mentioned, the set $HC(T)$ of hypercyclic vectors of a hypercyclic operator $T:X\to X$ is big in the sense of Baire. One can ask if/when this set is big enough to be \emph{spaceable}, that is, to contain an infinite dimensional closed subspace (but $0$). This is the origin of the concept of \emph{hypercyclic subspaces}, which is a topic well studied in linear dynamics nowadays. Whenever $T$ acts on a Banach algebra $X$, it is natural to ask if/when $HC(T)\cup\{0\}$ contains a subalgebra of $X$. Such a structure is called a \emph{hypercyclic algebra}. Sometimes the former contains a subalgebra that is not finitely generated. In this case we say that $HC(T)$ is \emph{algebrable}. When such a structure is also dense, we say that $HC(T)$ is \emph{dense-algebrable} (see \cite{aron2015lineability} for more details on the topic).

The first (negative) result on this topic  was obtained in 2007 by Aron, Conejero, Peris and Seoane-Sepúlveda \cite{aron2007powers}, who proved that no translation operator on the space of entire functions has a hypercyclic algebra. The first positive result is due to Bayart and Matheron \cite[Theorem 8.26]{bayartbook} and Shkarin \cite[Theorem 1.4]{shkarin2010}, who have shown independently that the operator of complex differentiation $D:f\mapsto f'$ acting on the space of entire functions $H(\CC)$ has a hypercyclic algebra. Bayart and Matheron's construction, which is based on a Baire argument, was later explored by other authors (see \cite{bes2018algebra}, \cite{bes2020algebrable}, \cite{bayart2019algebra} and \cite{besernst2020algebra}). In \cite{BCP}, expanding on ideas from \cite{bes2020algebrable}, a very general version of their criterion was obtained, which can be used to check the dense-algebrability of the set of hypercyclic vectors for an operator. We state it in full generality for the sake of completeness. We will need to establish the following notation. Let $X$ be an $F$-algebra, that is, a completely metrizable topological algebra. For $d\geq 1$, $u=(u_1,\dots, u_d)\in X^d$ and $\alpha=(\alpha_1,\dots,\alpha_d)\in\NN_0^d\backslash\{(0,\dots,0)\}$, we define $u^\alpha:=u_1^{\alpha_1}\cdots u_d^{\alpha_d}\in X$, with the convention that $u_j^0$ means that the term is omitted from the product. 

\begin{theorem}[\protect{\cite[Theorem 2.1]{BCP}}] 
 Let $T$ be a continuous operator on a separable commutative $F$-algebra $X$ and let $d\geq 1$. Assume that for any $P\subset \NN_0^d \backslash\{(0,\dots,0)\}$ finite and non-empty, 
   for any non-empty open subsets $U_1,\dots,U_d,V, W$ of $X$, with $0\in W$, 
 there exist $u=(u_1,\dots,u_d)\in U_1\times\cdots\times U_d$, $\beta\in P$ and $N\geq 1$ such that
 \begin{itemize}
     \item $T^N(u^\beta)\in V$,
     \item $T^N(u^\alpha)\in W $ for all $\alpha\in P$, $\alpha\neq\beta$.
 \end{itemize}
 Then the set of $d$-tuples that generate a hypercyclic algebra for $T$ is residual in $X^d$. Moreover, if the assumptions are satisfied for all $d\geq 1$, then $T$ admits a dense and not finitely generated hypercyclic algebra.
 \label{thm:generalcriterion}
\end{theorem}

\begin{figure}[H]
    \centering
    \includegraphics[width=0.5\linewidth]{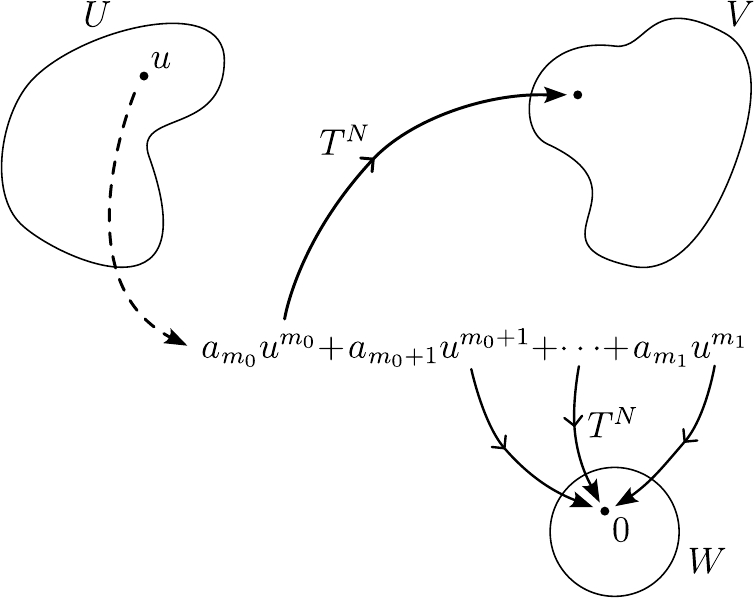}
    \caption{Graphical representation with $d=1$ and $\beta=\min A$.}
    \label{fig:trans-alg}
\end{figure}

In Figure \ref{fig:trans-alg}, one can see a graphical representation of this criterion with $d=1$, $A=\{m_0,\dots,m_1\}$, $m_0\leq m_1$, and $\beta=\min P$.

With this criterion at hand, the authors in \cite{BCP} have obtained characterizations of backward shifts supporting hypercyclic algebras on Fréchet sequence algebras equipped with either the coordinatewise or convolution product of sequences. In the case of Banach sequence algebras for which the canonical vectors form a Schauder basis, it is known that any hypercyclic backward shift has a dense and not finitely generated hypercyclic algebra. In this paper, we aim to extend these results to shifts on trees. We only consider the coordinatewise product on a tree, which can very naturally be defined by $(fg)(v)=f(v)g(v)$ for all $f,g\in \KK^A$ and $v\in A$. Thus, whenever we mention an algebra on $c_0$ or $\ell^p$-spaces of a tree, we are exclusively considering the coordinatewise product on these Banach algebras.  

It is worth noting that the conditions of the previous theorem can be seen as a strengthening of Birkhoff's transitivity theorem, which states that a given bounded linear operator $T$ defined on a separable Banach space $X$ is hypercyclic whenever it is topologically transitive, that is, for any pair of nonempty open subsets $U$ and $V$ of $X$, there exists a nonnegative integer $N$ such that $T^N(U) \cap V \neq \emptyset$. Furthermore, if the condition $T^k(U) \cap V \neq \emptyset$ holds for all $k\geq N$, then the operator $T$ is  called mixing.

The following simple proposition will be useful in the proof of Theorem \ref{thm:necessary:unrooted}. It is worth mentioning that the same technique was also used inside the proof of \cite[Theorem 5.1]{bayart2019algebra} for obtaining hypercyclic powers of vectors.

\begin{proposition}\label{prop:hc:powers}
    Let $T$ be a continuous operator on a separable $F$-algebra $X$. Suppose that, for all $m\geq 1$ and all pairs $(U,V)$ of non-empty open subsets of $X$, one can find $u\in U$ and $N\in\NN$ such that \[T^Nu^m\in V.\] Then there is a residual subset $\A$ of $X$ such that, for all $x\in \A$ and all $m\in\NN$, $x^m\in HC(T)$.
\end{proposition}
\begin{proof}
    Let $(V_k)_k$ be a basis of open sets for the topology of $X$. For each $k,m\geq 1$, define
    \[\A(m,k):=\{u\in X : \exists N\geq 1 \text{ such that }T^Nu^m\in V_k\}.\]
    Of course each $\A(m,k)$ is open in $X$, for $T$ is continuous. From the hypothesis, it follows that each $\A(m,k)$ is also dense in $X$. From the Baire Theorem we conclude that $\A:=\bigcap_{m,k\geq 1} \A(m,k)$ is a dense $G_\delta$-set. It is now easy to see that $\A$ is the desired residual set.
\end{proof}

The following simple lemma, extracted from \cite{BCP}, will be used multiple times throughout the text.
\begin{lemma}\label{lemma:maps}
Let $d\geq1$, let $P$ be a nonempty finite subset of $\NN_{0}^{d}\setminus\{(0,\ldots,0)\}$ and define, for each $\alpha=(\alpha_1,\ldots,\alpha_d)\in P$, the real linear map $L_\alpha:\RR^d\to\RR$ by
$L_\alpha(s)=\sum_{j=1}^d\alpha_js_j,$ $s\in\RR^d$. Then there exist $s\in(0,+\infty)^d$ and $\beta\in P$ such that $L_\alpha(s)>L_\beta (s)=1$, for all $\alpha\in P\backslash\{\beta\}$.
\end{lemma}

From now on, all Banach spaces considered will be over the field of complex numbers, that is, $\KK$ will denote $\CC$. Moreover, for any non-zero complex number $z$ and $s>0$, $z^s$ will denote one of the $s$-th powers of $z$.
 Furthermore, for any rooted tree $A$, a finite subset $F$ of $A$,  $f=\sum_{v\in F}f(v)e_v\in \CC^A$,  and for any $s>0$, we denote 
\[f^s=\sum_{v\in F} f(v)^s e_v.\]

\section{Shifts on rooted directed trees}\label{sec:rooted}

\subsection{Characterizations on $\ell^p(A)$ and $c_0(A)$}

We begin by proving our first main results, which characterize when weighted backward shifts on $\ell^p$-spaces, $1\leq p<\infty$, of rooted directed trees support hypercyclic algebras.
\begin{theorem}\label{cara:lp:hc-alg-rooted}
    Let $A$ be a rooted directed tree  and  $\lambda=(\lambda_v)_{v\in A}$ be a weight  such that its induced backward shift $B_\lambda:\ell^p(A)\to \ell^p(A)$ is bounded, with  $1\leq p<+\infty$. The following assumptions are equivalent.
\begin{enumerate}[$(i)$]
    \item $B_\lambda$ supports a dense, countably generated, free hypercyclic algebra.
    \item $B_\lambda$ supports a hypercyclic algebra.
    \item There exist an integer  $m\geq p$ and $f\in \ell^p(A)$ such that $f^m\in HC(B_\lambda)$.
    \item There is a sequence of positive integers $(n_k)_{k\geq1}$ such that, for every $v\in A$, 
    \[\dis\sup_{u\in\Chi^{n_k}(v)} |\lambda(v\to u)|\xrightarrow{k\to+\infty}+\infty.\]
\end{enumerate}
\end{theorem}
\begin{proof}
    It is clear that $(i)\Rightarrow (ii)\Rightarrow (iii)$. Let us show that $(iii)\Rightarrow (iv)$. Assume that $(iii)$ holds. Let $f\in\ell^p(A)$ satisfying $(iii)$ for $m\in\NN$ with $m\geq p$, $F\subset A$ finite and $\delta\in (0,\frac12)$. Since $\alpha f^m\in HC(B_\lambda)$ for all $\alpha\neq 0$, we can assume without loss of generality that $\|f\|_p<\delta^{1/m}$. Since $f^m$ is hypercyclic, there exists a positive integer $n\in\mathbb{N}$ such that $\Chi^n(A)\cap F=\varnothing$ and
\begin{equation}\label{exists-n}
    \bigg\|B_{\lambda}^{n}f^m-\sum_{u\in F}e_u\bigg\|_p<\delta.\end{equation}
From \eqref{exists-n}, we obtain, for every $v\in F$,
\begin{align*}
\dfrac{1}{2}<1-\delta<\big|(B_{\lambda}^{n}f^m)(v)\big|&\leq \sum_{u\in\Chi^n(v)}|\lambda(v\to u) f^m(u)| \\
&\leq \sup_{u\in\Chi^n(v)}|\lambda(v\to u)|\, \bigg(\sum_{u\in\Chi^n(v)}|f^m(u)|\bigg)\\
&\leq \sup_{u\in\Chi^n(v)}|\lambda(v\to u)|\,
\|f\|_m^m\\
&\leq \sup_{u\in\Chi^n(v)}|\lambda(v\to u)|\,
\|f\|_p^m\\
&\leq\delta \, \sup_{u\in\Chi^n(v)}|\lambda(v\to u)|,
\end{align*}
(notice that $m\geq p$ implies $\|\cdot\|_{m}\leq \|\cdot\|_p$). Therefore,
\[\sup_{u\in\Chi^n(v)}|\lambda(v\to u)|>\frac{1}{2\delta},\quad \forall v\in F.\]
To finish the proof, it is enough to apply the above arguments to an increasing sequence $(F_k)_{k\geq1}$ of a finite subset of $A$ such that $\bigcup_{k\geq1}F_k=A$ and to a positive sequence $(\delta_k)_{k\geq1}$ tending to zero. Then for each $k$, we find $n_k$ satisfying \eqref{exists-n}, which will define the sequence $(n_k)_{k\geq1}$ of positive integers as we wanted. 
    
Now, let us proceed to the proof of $(iv)\Rightarrow (i)$. Assume that $(iv)$ holds.   We shall apply Theorem \ref{thm:generalcriterion}. Let $d\geq 1$, let $P\subset \NN_0^d\backslash\{(0,\ldots,0)\}$ be finite and let $U_1,U_2,\ldots,U_d,V,W\subset \ell^p(A)$ be non-empty and open, with $0\in W$. We fix $(f_1,\dots, f_d)\in U_1\times \cdots\times U_d$ and $g\in V$ with support on some finite subset $F$ of $A$. We then apply Lemma \ref{lemma:maps} and we find $s\in (0,+\infty)^d$ and $\beta\in P$ such that the maps  $L_\alpha:\RR^d\to\RR$ defined $L_\alpha(s)=\sum_{j=1}^d\alpha_js_j$, for $s\in\RR^d$ and $\alpha\in P$, satisfy $1=L_\beta(s)<L_\alpha(s)$ for all $\alpha\in P\backslash\{\beta\}$.

Now, according to the hypothesis, there is $(n_k)_{k\geq1}$ an increasing sequence of positive entire numbers such that, for all $v\in A$,
    \[C_k(v):=\sup_{u\in \Chi^{n_k}(v)}\big|\lambda(v\to u)\big|\xrightarrow{k\to\infty} +\infty.\]
Thus, for each $a\in A$ and each $k\in\NN$, we can find $u_{a,k}\in \Chi^{n_k}(a)$ such that \begin{equation}\label{arafat}
    \frac{1}{|\lambda(a\to u_{a,k})|}< \frac{1}{C_k(a)}+\frac{1}{k}.
\end{equation}
We define $g_{a,k}:A\to \KK$ by \[g_{a,k}=\frac{1}{\lambda(a\to u_{a,k})}e_{u_{a,k}}\]
and we notice that
\begin{align*}
    (B_\lambda^{n_k}e_{u_{a,k}})(v)
    &=\sum_{u\in \Chi^{n_k}(v)}\lambda(v\to u)e_{u_{a,k}}(u) \\ &= \begin{cases} 0 &\text{ if } v\neq a\\ \lambda(a\to u_{a,k})  &\text{ if } v= a \end{cases}\\
    &=\lambda(a\to u_{a,k})e_{a}(v).
\end{align*}
Now, for all $k$ sufficiently big (conditions on the size of $k$ will be given during the proof), we set $h_k=(h_{k,1},\dots,h_{k,d})$, where 
\[h_{k,j}=f_j+\sum_{a\in F}\big(g(a)g_{a,k}\big)^{s_j},\quad j=1,\dots, d.\]
If $k\in\NN$ is big enough, we will have $\Chi^{n_k}(A)\cap F=\varnothing $. Thus, for each $\alpha=(\alpha_1,\ldots,\alpha_d)\in P$,
\begin{align*}
    h_k^\alpha = f_1^{\alpha_1}\cdots f_d^{\alpha_d}+\sum_{a\in F} \big(g(a)g_{a,k}\big)^{L_\alpha(s)}.
\end{align*}
Hence,
\begin{align*}
    B_\lambda^{n_k}h_k^\alpha
    &=\sum_{a\in F}g(a)^{L_\alpha(s)}B_\lambda^{n_k}g_{a,k}^{L_\alpha(s)} \\
    &=\sum_{a\in F}\frac{g(a)^{L_\alpha(s)}}{\lambda(a\to u_{a,k})^{L_\alpha(s)}}B_\lambda^{n_k}e_{u_{a,k}} \\
    &=\sum_{a\in F}\frac{g(a)^{L_\alpha(s)}}{\lambda(a\to u_{a,k})^{L_\alpha(s)-1}} e_{a}.
\end{align*}
If $\alpha=\beta$, we have 
\[B_\lambda^{n_k}h_k^\beta=\sum_{a\in F}g(a)e_a=g\in V;\]
and if $\alpha\in P\backslash\{\beta\}$, since $L_\alpha(s)-1>0$, from \eqref{arafat}, we find 
\begin{align*}
    \big\|B_\lambda^{n_k}h_k^\alpha\big\|_p
    &= \bigg( \sum_{a\in F} \frac{|g(a)|^{pL_\alpha(s)}}{|\lambda(a\to u_{a,k})|^{p(L_\alpha(s)-1)}} \bigg)^{1/p}\xrightarrow{k\to\infty} 0.
\end{align*}
Thus, $B_\lambda^{n_k}h_k^\alpha\in W$ if $k$ is big enough. Finally, still from \eqref{arafat}, we get
\begin{align*}
    \|h_{k,j}-f_j\|_p&=\bigg(
    \sum_{a\in F}\frac{|g(a)|^{ps_j}}{|\lambda(a\to u_{a,k})|^{ps_j}} \bigg)^{1/p}\xrightarrow{k\to\infty} 0.
\end{align*}
Hence, if $k$ is big enough, $h_{k,j}\in U_j$ for all $j=1,\dots,d$. This completes the proof.
\end{proof}
As a consequence of the previous theorem combined with Theorem \ref{carac-l1}, we deduce the following corollary, which states that weighted backward shifts on $\ell^1$-spaces of rooted directed trees support hypercyclic algebras whenever they are hypercyclic.
\begin{corollary}\label{hc_alg_l1_rooted}
    Let $A$ be a rooted directed tree and let $\lambda=(\lambda_u)_{u\in A}$ be a weight such that the induced weighted backward shift $B_\lambda$ is bounded on $\ell^1(A)$. The following assumptions are equivalent.
\begin{enumerate}[$(i)$]
    \item $B_\lambda$ is hypercyclic.
    \item $B_\lambda$ supports a hypercyclic algebra.
    \item $B_\lambda$ supports a dense, countably generated, free hypercyclic algebra.
\end{enumerate}
\end{corollary}
Notice that condition $(iv)$ in Theorem \ref{cara:lp:hc-alg-rooted}, being equivalent to the hypercyclicity of $B_\lambda$ in $\ell^1(A)$, implies that, for $B_\lambda$ to admit a hypercyclic algebra on $\ell^p(A), 1<p<+\infty,$ it is required that $B_\lambda$ is hypercyclic on $\ell^1(A)$. This will be further discussed in the next section. 

We will now establish the same equivalences as the previous corollary for weighted backward shift operators on the $c_0$-space of rooted directed trees. Differently than the previous case of $\ell^p$-spaces, now we make use of \eqref{key-lemma} for constructing the right inverse operator of backward shifts.

\begin{theorem}\label{hc_alg_c0_rooted}
Let $A$ be a rooted directed tree and let $\lambda=(\lambda_u)_{u\in A}$ be a weight such that the induced backward shift $B_\lambda$ is bounded on $c_0(A)$. The following assumptions are equivalent.
\begin{enumerate}[$(i)$]
    \item $B_\lambda$ is hypercyclic.
    \item $B_\lambda$ supports a hypercyclic algebra.
    \item $B_\lambda$ supports a dense, countably generated, free hypercyclic algebra.
\end{enumerate}
\end{theorem}
\begin{proof}
The implications $(iii)\Rightarrow(ii)\Rightarrow(i)$ are immediate. In order to prove $(i)\Rightarrow(iii)$, let us suppose that $B_\lambda$ is hypercyclic. Once more, we aim to apply Theorem \ref{thm:generalcriterion}. The initialization is the same as in the case of $\ell^p(A)$ (see the first paragraph in the proof of the implication $(iv)\Rightarrow (i)$ of Theorem \ref{cara:lp:hc-alg-rooted}). Now, since $B_\lambda$ is hypercyclic, from Theorem \ref{carac-c0}, there is an increasing sequence of positive entire numbers $(n_k)_{k\geq1}$ such that, for all $v\in A$,
\[\sum_{u\in \Chi^{n_k}(v)}\big|\lambda(v\to u)\big|\xrightarrow{k\to+\infty}+\infty.\]
For each $v\in A$ and $k\in \NN$, we apply  \eqref{key-lemma} with $J=\Chi^{n_k}(v)$ and we get the existence of $g_{v,k}:A\to \KK$ with finite support in $\Chi^{n_k}(v)$ satisfying
\begin{align}\label{champignon}
\|g_{v,k}\|_1=1\quad\text{and}\quad\sup_{u\in\Chi^{n_k}(v)}\frac{|g_{v,k}(u)|}{|\lambda(v\to u)|}\xrightarrow{k\to+\infty}0.
\end{align}
For each $v\in A$ and $k\in\NN$, we define \[R_{v,k} = \sum_{u\in\Chi^{n_k}(v)} \frac{|g_{v,k}(u)|}{\lambda(v\to u)}e_u\]
and, by using \eqref{champignon}, we notice that
\begin{align*}
    B_{\lambda}^{n_k}R_{v,k}
    &=\sum_{u\in\Chi^{n_k}(v)} \frac{|g_{v,k}(u)|}{\lambda(v\to u)}B_{\lambda}^{n_k}e_u\\
    &=\sum_{u\in\Chi^{n_k}(v)} \frac{|g_{v,k}(u)|}{\lambda(v\to u)}\lambda(v\to u)e_v\\
    &=\sum_{u\in\Chi^{n_k}(v)} |g_{v,k}(u)|e_v=e_v.
\end{align*}
Hence, choosing $k\in\NN$ big enough so that $\Chi^{n_k}(A)\cap F=\varnothing $, we define $h_k=(h_{k,1},\ldots,h_{k,d})$ with
\[h_{k,j}=f_j+\sum_{a\in F}g(a)^{s_j}R_{a,k}^{s_j}, \quad j=1,\dots,d,\]
and we get
\begin{align*}
    B_{\lambda}^{n_k}h_k^{\beta}= g\in V.
\end{align*}
We also have, for each $a\in F$ and $j=1,\dots, d$,
\begin{align*}
    \|R_{a,k}^{s_j}\|_{\infty} &= \sup_{v\in A} |R_{a,k}^{s_j}(v)| \\
    &= \sup_{v\in \Chi^{n_k}(a)} |R_{a,k}^{s_j}(v)|\\
    &= \bigg(\sup_{v\in \Chi^{n_k}(a)}\frac{|g_{a,k}(v)|}{|\lambda(a\to v)|}\bigg)^{s_j}\xrightarrow{k\to +\infty}0,
\end{align*}
where the last limit comes from \eqref{champignon}. Thus, $h_k\in U$ if $k$ is big enough. Now, we consider $\alpha \in P\backslash\{\beta\}$ and $k\in\NN$, and we calculate
\begin{align*}
\|B_{\lambda}^{n_k}h_k^{\alpha}\|_{\infty}
&= \bigg\|\sum_{a\in F}g(a)^{L_\alpha(s)}B_\lambda^{n_k} R_{a,k}^{L_\alpha(s)}\bigg\|_{\infty}\\
&\leq\max_{a\in F}\sum_{u\in \Chi^{n_k}(a)}|g(a)|^{L_\alpha(s)}\frac{|g_{a,k}(u)|^{L_\alpha(s)}}{|\lambda(a\to u)|^{L_\alpha(s)-1}}\\
&= \max_{a\in F}\sum_{u\in \Chi^{n_k}(a)}|g(a)|^{L_\alpha(s)}|g_{a,k}(u)|\bigg(\frac{|g_{a,k}(u)|}{|\lambda(a\to u)|}\bigg)^{L_\alpha(s)-1}\\
&\leq \max_{a\in F}|g(a)|^{L_\alpha(s)}\sup_{u\in\Chi^{n_k}(a)}\bigg(\frac{|g_{a,k}(u)|}{|\lambda(a\to u)|}\bigg)^{L_\alpha(s)-1}\sum_{u\in \Chi^{n_k}(a)}|g_{a,k}(u)|\\
&=\max_{a\in F}|g(a)|^{L_\alpha(s)}\sup_{u\in\Chi^{n_k}(a)}\bigg(\frac{|g_{a,k}(u)|}{|\lambda(a\to u)|}\bigg)^{L_\alpha(s)-1}\xrightarrow{k\to+\infty}0,
\end{align*}
where the last limit follows again from \eqref{champignon}. Therefore, if $k$ is big enough, we have $B_{\lambda}^{n_k}h_k^{\alpha}\in W$ for all $\alpha\in P\backslash\{\beta\}$. From Theorem \ref{thm:generalcriterion}, $B_\lambda$ has a dense, countably generated, free hypercyclic algebra on $c_0(A)$.
\end{proof}

\subsection{Non-existence of hypercyclic algebras on $\ell^p(A), 1<p<+\infty$}

As we saw in the last subsection, weighted backward shifts on the spaces $\ell^1$ and $c_0$ support hypercyclic algebras whenever they are hypercyclic. However, this equivalence does not hold on $\ell^p$-spaces with $1<p<\infty$, for there are trees $A$ supporting shifts that are hypercyclic on $\ell^p(A)$ but not in $\ell^1(A)$. Simple examples are Rolewicz operators on rooted directed trees. Indeed, for the classical family of Rolewicz operators $\lambda B := B_{(\lambda)_{v\in V}}$, where $\lambda \in \KK \setminus \{0\}$, by using Theorem \ref{cara:lp:hc-alg-rooted} we obtain the following readable criterion.

\begin{corollary}
A Rolewicz operator $\lambda B$ on an $\ell^p$-space, with $1\leq p<\infty$, of a rooted directed tree supports a hypercyclic algebra if, and only if, $|\lambda|>1$.
\end{corollary}

Notice that, according to Proposition \ref{boundedness}, Rolewicz operators are bounded on $\ell^p(A)$ for $1 < p < \infty$ whenever $\sup_{v\in A}|\Chi(v)| < \infty$, where $|\Chi(v)|$ denotes the cardinality of the set $\Chi(v)$.

The previous corollary provides, by itself, many examples of hypercyclic operators without hypercyclic algebras on $\ell^p$-spaces where $p$ is greater than 1. In fact, let $N>1$ in $\NN$ and let $A$ be the $N$-adic rooted tree, that is, each vertex in $A$ has exactly $N$ children. Then $\lambda B$ is hypercyclic (even mixing) on $\ell^p(A)$ if and only if $|\lambda| > N^{-1/p*}$ (see \cite[Corollary 4.5]{Karl1}). These operators are also chaotic (\cite[Theorem 9.4]{Karl2}). Another counter example in the dyadic rooted tree can be obtained as follows.
\begin{example}
Let $1<p<+\infty$ and let $A$ be the dyadic rooted tree. Then there are $m_0\in\NN$ and a hypercyclic (even mixing) weighted backward shift $B_\lambda$ on $\ell^p(A)$ with no hypercyclic algebra, and which satisfies the stronger property \[\sup_{n\geq 1} \sum_{u\in\Chi^n(v)}|\lambda(v\to u)|^{m/p}<+\infty\]
for all integers $m\geq m_0$.
\end{example}
\begin{figure}[H]
    \centering
    \includegraphics[width=0.4\textwidth]{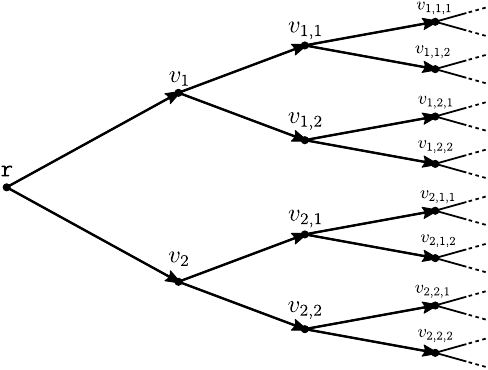}
    \caption{Lexicographical order on the dyadic partition}
    \label{fig:dyadic}
\end{figure}
\begin{proof}
We note $I_2=\{1,2\}$ and we define $\Chi(\root)=\{v_1, v_2\}$ and, for any $n\in\NN$ and any ${\bf i}\in I_2^n$, $\Chi(v_{\bf i})=\{v_{{\bf i},1}, v_{{\bf i},2}\}$ (see Figure \ref{fig:dyadic}). For each $n\in\NN$, let $\theta_n({\bf i})$ denote the position of the index ${\bf i}\in I_2^n$ in the lexicographical order. Since $p>1$, we can fix $m_0\in \NN$ such that $\frac{p}{m_0}<\frac{p-1}{p}$. We choose $\alpha\in(\frac{p}{m_0}, \frac{p-1}{p})$ and we define
\[
\lambda_{v_{\bf i}}=
\frac{1}{\theta_n({\bf i})^{\alpha}\lambda(\root\to \Par({v_{\bf i}}))}, 
\quad \forall n\in\NN\text{ and } {\bf i}\in I_2^n.
\]
We first verify that the induced weighted backward shift $B_\lambda$ is bounded on $\ell^p(A)$. In fact, for each $v\in A$, say $v=v_{\bf i}$ for some $n\in \NN$ and ${\bf i}\in I_2^n$, we have
\begin{align}
\sum_{u\in \Chi(v)}|\lambda_u|^{p^*}&=|\lambda_{v_{{\bf i},1}}|^{p^*}+|\lambda_{v_{{\bf i},2}}|^{p^*} \nonumber \\
&= \frac{1}{\theta_{n+1}({\bf i},1)^{\alpha p^*}\lambda(\root\to {v_{\bf i}})^{p^*}}+\frac{1}{\theta_{n+1}({\bf i},2)^{\alpha p^*}\lambda(\root\to {v_{\bf i}})^{p^*}}  \nonumber \\
&= \label{halal} \frac{1}{\big(\theta_{n+1}({\bf i},2)-1\big)^{\alpha p^*}\lambda(\root\to {v_{\bf i}})^{p^*}}+\frac{1}{\theta_{n+1}({\bf i},2)^{\alpha p^*}\lambda(\root\to {v_{\bf i}})^{p^*}}
\end{align} 
We note that, for all $n\in \NN$ and ${\bf i}\in I_2^n$,
\begin{align*}
\lambda(\root\to v_{\bf i}) 
&=\lambda(\root\to \Par(v_{\bf i}))\lambda_{v_{\bf i}} \\
&= \lambda(\root\to \Par(v_{\bf i}))\frac{1}{\theta_n({\bf i})^{\alpha}\lambda(\root\to \Par({v_{\bf i}}))}\\
&=\frac{1}{\theta_n({\bf i})^\alpha}.
\end{align*}
From the definition of the lexicographical order we see that
\[\lambda(\root\to v_{\bf i})=\frac{1}{\theta_n({\bf i})^\alpha} = \bigg(\frac{2}{\theta_{n+1}({\bf i},2)}\bigg)^\alpha.\]
Therefore, from \eqref{halal},
\begin{align*}
\sum_{u\in \Chi(v)}|\lambda_u|^{p^*}
= \bigg(\frac{\theta_{n+1}({\bf i},2)}{2\big(\theta_{n+1}({\bf i},2)-1\big)}\bigg)^{\alpha p^*}+\bigg(\frac{\theta_{n+1}({\bf i},2)}{2\theta_{n+1}({\bf i},2)}\bigg)^{\alpha p^*}
\leq 1+\bigg(\frac{1}{2}\bigg)^{\alpha p^*},
\end{align*} 
thus \[\sup_{v\in A} \sum_{u\in \Chi(v)}|\lambda_u|^{p^*} < +\infty.\]
Hence,  by Proposition \ref{boundedness}, $B_\lambda$ is bounded on $\ell^p(A)$.

Now, we see that, for all $N\in\NN$ and all $v_{\bf i}\in A$, with $n\in \NN$ and ${\bf i}\in I_2^n$, we have $\Chi^N(v_{\bf i})=\{v_{{\bf i},{\bf j}} : {\bf j}\in I_2^{N}\}$. Hence,
\begin{align*}
\sum_{u\in \Chi^{N}(v_{\bf i})} |\lambda(v_{\bf i}\to u)|^{p^*} 
    &= \sum_{{\bf j}\in I_2^{N}} \bigg(\frac{\lambda(\root\to v_{{\bf i},{\bf j}})}{\lambda(\root\to v_{\bf i})}\bigg)^{p^*}\\
    &=\sum_{{\bf j}\in I_2^{N}} \bigg(\frac{\theta_n({\bf i})}{\theta_{n+N}({\bf i},{\bf j})}\bigg)^{\alpha p^*} \\
    &= \sum_{k=1}^{2^N} \bigg(\frac{\theta_n({\bf i})}{2^N(\theta_n({\bf i})-1)+k}\bigg)^{\alpha p^*}\\
    &\geq 2^N\bigg(\frac{\theta_n({\bf i})}{2^N(\theta_n({\bf i})-1)+2^N}\bigg)^{\alpha p^*}\\
    &= 2^{N(1-\alpha p^*)} \xrightarrow{N\to +\infty}+\infty.
\end{align*}
(One can check the third equality through a simple induction on $N$). From Theorem \ref{carac-lp}, we conclude that $B_\lambda$ is hypercyclic. On the other hand, since $\lambda_v\leq 1$ for all $v\in A$, as soon as $m\geq m_0$ we have
\begin{align*}
\sum_{u\in \Chi^{N}(v_{\bf i})} |\lambda(v_{\bf i}\to u)|^{\frac{m}{p}} 
    &\leq \sum_{u\in \Chi^{N}(v_{\bf i})} |\lambda(v_{\bf i}\to u)|^{\frac{m_0}{p}}\\
    &= \sum_{k=1}^{2^N} \bigg(\frac{\theta_n({\bf i})}{2^N(\theta_n({\bf i})-1)+k}\bigg)^{\frac{\alpha m_0}{p}}\\
    &\leq \sum_{k=1}^{2^N} \bigg(\frac{\theta_n({\bf i})}{k(\theta_n({\bf i})-1)+k}\bigg)^{\frac{\alpha m_0}{p}}\\
    &\leq \sum_{k=1}^{+\infty} \bigg(\frac{1}{k}\bigg)^{\frac{\alpha m_0}{p}}.
\end{align*}
Therefore, for all $v\in A$,
\[\sup_{N\geq 1} \sum_{u\in \Chi^{N}(v)} |\lambda(v\to u)|^{\frac{m}{p}} <+\infty\] as we wanted. In particular, it follows from Theorem \ref{cara:lp:hc-alg-rooted} that $B_\lambda$ does not have a hypercyclic algebra.
\end{proof}

The previous example shows that having a hypercyclic algebra on $\ell^p$-spaces of $N$-adic trees is rather ``sensible'' to powers. The stronger conclusion is almost the condition for hypercyclicity on $\ell^p$, but the power ruins the existence of algebras.

Despite this difficulty, we can still prove the equivalence between being hypercyclic and having a hypercyclic algebra on a big class of weighted backward shifts on $\ell^p$-spaces of a particular tree.

\begin{figure}[H]
    \centering
    \includegraphics[width=8cm]{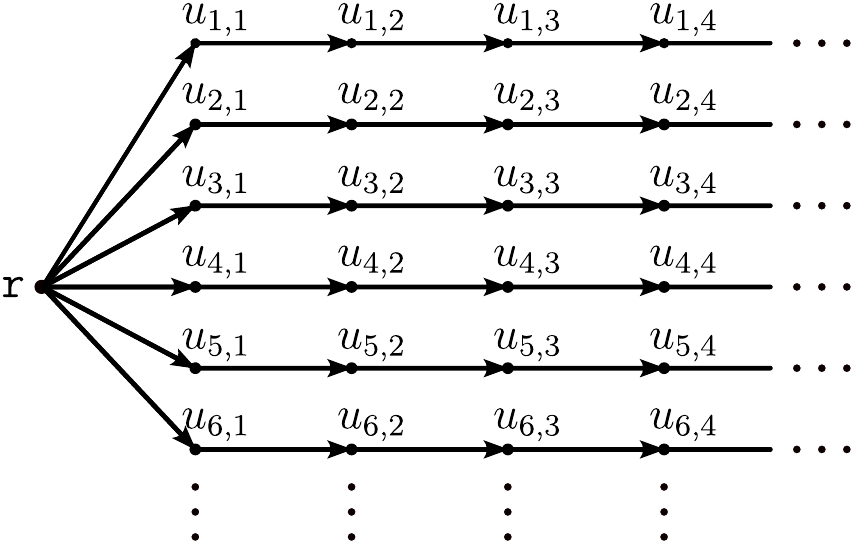}
    \caption{Directed tree in Example \ref{menthe}}
    \label{fig:menthe}
\end{figure}
\begin{example}\label{menthe}
    Let us consider the rooted tree $A$, represented in Figure \ref{fig:menthe} and defined as follows. 
\[\begin{cases}
    \Chi(\root)=\{u_{i,1} : i\in \NN\},\\
    \Chi({u_{i,j}})=\{u_{i,j+1}\}, \ \forall i, j\in\NN.
\end{cases}\]
We consider a weight $\lambda=(\lambda_u)_{u\in A}$ such that $B_\lambda$ is bounded on $\ell^p(A), 1<p<+\infty,$ and satisfying 
\[\lambda(\root\to u_{i,j}) = \alpha_i\beta_j,\]
where $(\alpha_i)_i$ and $(\beta_j)_j$ are sequences of complex scalars. Then $B_\lambda$ is hypercyclic on $\ell^p(A)$ if and only if it has a hypercyclic algebra.\end{example}
\begin{proof}
The reciprocal is trivial. Let us suppose $B_\lambda$ is hypercyclic.  From Theorem \ref{carac-lp}, there exists an increasing sequence of positive integers $(n_k)_{k\in\NN}$ such that, for every $v\in A$,
\begin{equation}\label{raisin-sec}
    \sum_{u\in \Chi^{n_k}(v)}\big|\lambda(v\to u)\big|^{p^*} \xrightarrow{k\to+\infty}+\infty.
\end{equation}
By applying this condition to $v=\root$, we get
\begin{align*}
    |\beta_{n_k}|^{p^*}\sum_{i=1}^{+\infty}|\alpha_{i}|^{p^*}
    &=\sum_{i=1}^{+\infty}\big|\lambda(\root\to u_{i,n_k})\big|^{p^*}\\
    &=\sum_{u\in \Chi^{n_k}(\root)}\big|\lambda(\root\to u)\big|^{p^*} \xrightarrow{k\to+\infty}+\infty.
\end{align*}
From the continuity of $B_\lambda$, we see that $(\alpha_i)_i\in \ell^{p^*}(\NN)$. In fact, we might have
\[|\beta_1|^{p^*}\sum_{i=1}^{+\infty}|\alpha_i|^{p^*}=\sum_{i=1}^{+\infty} |\lambda_{u_{i,1}}|^{p^*}=\sum_{u\in\Chi(\root)}|\lambda_{u}|^{p^*}<+\infty.
\]
Hence,
\begin{equation}\label{uva-r}
|\beta_{n_k}|\xrightarrow{k\to+\infty}+\infty.
\end{equation}
Similarly, by applying \eqref{raisin-sec} to $v=u_{i,j}$, for some $i,j\in\NN$, we get
\begin{equation}\label{uva-ij}
\big|\lambda(u_{i,j}\to u_{i,j+n_k})\big|^{p^*}
    =\sum_{u\in \Chi^{n_k}(v)}\big|\lambda(v\to u)\big|^{p^*} \xrightarrow{k\to+\infty}+\infty.
\end{equation}
It is now easy to see that conditions \eqref{uva-r} and \eqref{uva-ij} allow us to verify the condition (iv) in Theorem \ref{cara:lp:hc-alg-rooted}.
\end{proof}

\section{Shifts on unrooted directed trees}\label{sec:unrooted}

In this section we deal with shifts $B_\lambda$ on unrooted directed trees. As expected, and just like with hypercyclicity, it is much more difficult to deal with this case when looking for hypercyclic algebras. Before proceeding to our first main result, let us discuss the set of generators of a hypercyclic algebra.

It is a well known fact that, whenever an operator acting on an $F$-space is hypercyclic, then the set of hypercyclic vectors is a dense $G_\delta$-set, what can be easily proven through a Baire argument. However, it is not known whether an analogous statement is true for hypercyclic algebras or not. In fact, the best known technique to produce hypercyclic algebras also uses a Baire argument (see Theorem \ref{thm:generalcriterion}), thus giving a dense $G_\delta$-set of elements generating them. It is not known whether this is always the case or not. Therefore, in the current state of development of the theory, it is natural to assume the existence of a dense set of vectors generating hypercyclic algebras whenever one needs such hypothesis. This can be particularly useful when trying to obtain necessary conditions for the existence of hypercyclic algebras. The existence of a dense hypercyclic algebra is an even stronger assumption, but will be sometimes used for making statements more simple and pleasing to read.

The following theorem gives a necessary condition for weighted backward shifts on $\ell^p$-spaces of unrooted directed trees to support dense sets of generators of hypercyclic algebras.
\begin{theorem}\label{thm:necessary:unrooted}
    Let $A$ be an unrooted directed tree and $\lambda=(\lambda_v)_{v\in A}$ be a weight such that its induced backward shift $B_\lambda:\ell^p(A)\to \ell^p(A)$ is bounded, with $1\leq p<\infty$. Consider the following assertions.
    \begin{enumerate}[$(i)$]
        \item $B_\lambda$ supports  a dense hypercyclic algebra. 
        \item There is a dense set of elements $f\in \ell^p(A)$ generating a hypercyclic algebra for $B_\lambda$.
        \item There is a residual set of elements $f\in \ell^p(A)$ such that $f^m\in HC(B_\lambda)$ for all $m\geq 1$.
        \item There is a dense set of elements $f\in\ell^p(A)$ such that $f^m\in  HC(B_\lambda)$ for some integer $m \geq p$.
        \item There is a sequence of non-negative integers $(n_k)_{k\geq1}$ such that, for every $v\in A$,
    \begin{align*}
        \sup_{u\in \Chi^{n_k}(v)}|\lambda(v\to u)| \xrightarrow{k\to+\infty} +\infty,
    \end{align*}
    and
    \[\max\Bigg(\frac{1}{|\lambda(\Par^{n_k}(v)\to v)\big|},\sup_{u\in \Chi^{n_k}(\Par^{n_k}(v))}\frac{\big|\lambda(\Par^{n_k}(u)\to u)\big|}{|\lambda(\Par^{n_k}(v)\to v)\big|}\Bigg)\xrightarrow{k\to+\infty}+\infty.\]
    \end{enumerate}
    Then $(i)\Rightarrow (ii)\Rightarrow (iii)\Leftrightarrow (iv)\Leftrightarrow (v)$.
\end{theorem}
\begin{proof}
    It is clear that $(i)\Rightarrow (ii)\Rightarrow (iii)\Rightarrow (iv)$. Let us show that $(iv)\Rightarrow (v)$.  Let $F\subset A$ finite, $N\geq 1$ and $0<\veps<\frac{1}{2}$. We choose some $\delta\in(0,\veps)$ such that, for any $h\in \ell^p(A)$, $\|h-\sum_{v\in F}e_v\|_p<\delta$ implies $|h(v)|\geq \frac{1}{2}$ for all $v\in F$. We define a subset $G\subset F$ by picking exactly one vertex from $F$ per generation. By the hypothesis, there are $f\in \ell^p(A)$, $n\geq N$ and $m\geq p$ such that 
\begin{equation}\label{alg0}
    \bigg\|f-2\,\sum_{v\in G}e_v\bigg\|_p<\veps \quad \text{and}\quad \bigg\|B_\lambda^nf^m-\sum_{v\in F}e_v\bigg\|_p<\delta.
\end{equation}
By taking $N$ bigger if necessary, we can assume that 
\begin{equation}\label{alg1}
    \Chi^n(F)\cap G=\varnothing, \quad \Par^n(G)\cap F=\varnothing \quad \text{and}\quad \Par^n(u)=\Par^n(v), \ \forall u\sim v\in F.
\end{equation}
Let us show that, for all $v\in F$, 
\begin{equation}\label{abb1}
   \sup_{u\in\Chi^n(v)}|\lambda(v\to u)|>\frac{1}{2\varepsilon},
\end{equation}
and, for all $v\in G$, 
\begin{equation}\label{left_cond}
\max\bigg(\frac{1}{|\lambda(\Par^{n}(v)\to v)|},\sup_{u\in \Chi^{n}(\Par^{n}(v))}\frac{|\lambda(\Par^{n}(u)\to u)|}{|\lambda(\Par^{n}(v)\to v)|}\bigg)\geq \frac{1}{2\veps}.
\end{equation} 
By using \eqref{alg0} and \eqref{alg1}, we have $\|f\chi_{\Chi^n(v)}\|_p<\varepsilon$, for every $v\in F$. Moreover, by using again  \eqref{alg0}, we obtain that, for every $v\in F$,
\begin{align*}
\dfrac{1}{2}<1-\delta<\big|(B_{\lambda}^{n}f^m)(v)\big|&\leq \sum_{u\in\Chi^n(v)}|\lambda(v\to u) f^m(u)| \\
&\leq \sup_{u\in\Chi^n(v)}|\lambda(v\to u)|\, \Big(\sum_{u\in\Chi^n(v)}|f^m(u)|\Big)\\
&\leq \sup_{u\in\Chi^n(v)}|\lambda(v\to u)|\,
\|f\chi_{\Chi^n(v)}\|_m^m\\
&\leq \sup_{u\in\Chi^n(v)}|\lambda(v\to u)|\,
\|f\chi_{\Chi^n(v)}\|_p^m\\
&<\varepsilon^m \, \sup_{u\in\Chi^n(v)}|\lambda(v\to u)|
\end{align*}
(notice that $m\geq p$ implies $\|\cdot\|_{m}\leq \|\cdot\|_p$). Therefore,
\[\sup_{u\in\Chi^n(v)}|\lambda(v\to u)|>\frac{1}{2\varepsilon},\quad \forall v\in F.\]
Let us show that \eqref{left_cond} holds. Let us fix $v\in G$. We have
\begin{align*}
    (B_\lambda^nf^m)\big(\Par^n(v)\big)
        &=\sum_{u\in \Chi^n(\Par^n(v))} \lambda(\Par^n(v)\to u)f^m(u) \\
        &= \lambda(\Par^n(v)\to v)f^m(v)+\sum_{u\in \Chi^n(\Par^n(v))\setminus\{v\}} \lambda(\Par^n(v)\to u)f^m(u) \\
        &=\lambda(\Par^n(v)\to v)\bigg(f^m(v)+\sum_{u\in \Chi^n(\Par^n(v))\setminus\{v\}}\frac{\lambda(\Par^n(v)\to u)}{\lambda(\Par^n(v)\to v)}\, f^m(u)\bigg).
\end{align*}
From \eqref{alg1} we know that $\Par^n(v)\notin F$, so from \eqref{alg0} we get
\[|\lambda(\Par^n(v)\to v)|\cdot \bigg|f^m(v)+\sum_{u\in \Chi^n(\Par^n(v))\setminus\{v\}}\frac{\lambda(\Par^n(v)\to u)}{\lambda(\Par^n(v)\to v)}\, f^m(u)\bigg|
<\delta< \veps.\]
We can assume that $|\lambda(\Par^n(v)\to v)|>2\veps$ (otherwise, the conclusion \eqref{left_cond} is trivial). Thus
\[ \bigg|f^m(v)+\sum_{u\in \Chi^n(\Par^n(v))\setminus\{v\}}\frac{\lambda(\Par^n(v)\to u)}{\lambda(\Par^n(v)\to v)}\, f^m(u)\bigg|<\frac{1}{2}.\]
By using \eqref{alg0} and \eqref{alg1}, we have $\|f\chi_{\Chi^n(\Par^n(v))\setminus\{v\}}\|_p<\varepsilon$. Therefore, 
\begin{align*}
    \ |f(v)|^m -\dfrac{1}{2}&<\sum_{u\in \Chi^n(\Par^n(v))\setminus\{v\}}\bigg|\frac{\lambda(\Par^n(v)\to u)}{\lambda(\Par^n(v)\to v)}\, f^m(u)\bigg|\\
    &\leq \sup_{u\in \Chi^n(\Par^n(v))}\frac{|\lambda(\Par^n(v)\to u)|}{|\lambda(\Par^n(v)\to v)|} \bigg( \sum_{u\in \Chi^n(\Par^n(v))\setminus\{v\}}|f(u)|^m\bigg)\\
    &\leq \sup_{u\in \Chi^n(\Par^n(v))}\frac{|\lambda(\Par^n(v)\to u)|}{|\lambda(\Par^n(v)\to v)|}\|f\chi_{\Chi^n(\Par^n(v))\setminus\{v\}}\|_{m}^{m}\\
    &\leq \sup_{u\in \Chi^n(\Par^n(v))}\frac{|\lambda(\Par^n(v)\to u)|}{|\lambda(\Par^n(v)\to v)|}\|f\chi_{\Chi^n(\Par^n(v))\setminus\{v\}}\|_{p}^{m}\\
    &< \varepsilon^m  \sup_{u\in \Chi^n(\Par^n(v))}\frac{|\lambda(\Par^n(v)\to u)|}{|\lambda(\Par^n(v)\to v)|}.
\end{align*}
By using again \eqref{alg0}, we have $|f(v)|^m>(2-\varepsilon)^m$, hence
\[\sup_{u\in \Chi^n(\Par^n(v))}\frac{|\lambda(\Par^n(v)\to u)|}{|\lambda(\Par^n(v)\to v)|}>\dfrac{(2-\varepsilon)^m-1/2}{\varepsilon^m}>\dfrac{1}{2\varepsilon},\]
this shows that \eqref{left_cond} also holds. To get the desired conclusion, it is enough to apply the above  arguments to an increasing sequence  $(F_k)_{k\geq1}$ of finite subsets of $A$ such that $\bigcup_{k\geq1}F_k=A$, and a sequence  $(\delta_k)_{k\geq1}$  of positive numbers tending to zero. For each $k$, we can find a positive integer $n_k$ and a subset $G_k \subset F_k$ that satisfy conditions \eqref{abb1} and \eqref{left_cond}. Additionally, the sequence $(G_k)_{k \in \mathbb{N}}$ can be chosen to be increasing. It is clear that the first limit in condition (v) holds for any $v \in A$. Now, let us verify the second limit in (v). For a fixed \(v \in A\), there exists some $k_0 \in \mathbb{N}$ such that $v \in F_k$ for all $k \geq k_0$. By the definition of the sets $G_k$,  we can find $v_0 \in G_k$, for every $k \geq k_0$, such that $v \sim v_0$. Let $m_0 \in \mathbb{N}$ be such that $\Par^{m_0}(v) = \Par^{m_0}(v_0)$. Using condition \eqref{left_cond}, for every $k \geq k_0$ such that $n_k > m_0$, we obtain:

\begin{align*}
  M_k(v):&=\max\bigg( \frac{1}{|\lambda(\Par^{n_k}(v)\to v)|},\sup_{u\in \Chi^{n_k}(\Par^{n_k}(v))}\frac{|\lambda(\Par^{n_k}(u)\to u)|}{|\lambda(\Par^{n_k}(v)\to v)|}\bigg)\\
  &=\frac{|\lambda(\Par^{m_0}(v_0)\to v_0)|}{|\lambda(\Par^{m_0}(v)\to (v))|}\, M_k(v_0)\\
  &\geq \frac{|\lambda(\Par^{m_0}(v_0)\to v_0)|}{|\lambda(\Par^{m_0}(v)\to (v))|} \,\frac{1}{2\veps_k} \xrightarrow{k\to+\infty}+\infty,
\end{align*}
hence the second limit of $(v)$ holds.
 
We finish the proof by showing that $(v)\Rightarrow (iii)$. We aim to apply Proposition \ref{prop:hc:powers}, so let  $U,V$ be non-empty open subsets of $\ell^p(A)$ and let $m\geq 1$. We fix $f\in U$ and $g\in V$ with support on some finite subset $F$ of $A$. For any $a\in A$ and $k\geq 1$, there exist $u_{(a,n_k)}\in \Chi^{n_k}(a)$ and $b_{(a,n_k)}\in  \Chi^{n_k}(\Par^{n_k}(a))$  such that
    \begin{equation*}
      |\lambda(a\to u_{(a,n_{k})})|\xrightarrow{k\to+\infty} +\infty,
    \end{equation*}
    and
    \begin{equation*}\label{lp:unrooted:eq2}
      \max\bigg(\dfrac{1}{|\lambda\big(\Par^{n_k}(a)\to a\big)\big|},\dfrac{\big|\lambda\big(\Par^{n_k}(a)\to b_{(a,n_k)}\big)\big|}{\big|\lambda\big(\Par^{n_k}(a)\to a\big)\big|}\bigg)\xrightarrow{k\to+\infty}+\infty.
    \end{equation*}    
     By extracting the vertices depending on their behavior, we can write  $F= F_{1} \cup F_{2}$, with $F_{1} \cap F_{2} = \emptyset$, such that 
     \begin{equation}\label{lp:unrooted:eq1}
      |\lambda(a\to u_{(a,n_{k})})|\xrightarrow{k\to+\infty} +\infty, \forall a\in F,
    \end{equation}
    \begin{equation}\label{F_j^1:unrooted:eq}
    \lambda\big(\Par^{n_{k}}(a)\to a\big)\xrightarrow{k\to+\infty}0, \quad \forall a\in F_{1},
\end{equation}
     and
\begin{equation}\label{F_j^2:unrooted:eq}
      \dfrac{\lambda\big(\Par^{n_k}(a)\to a\big)}{\lambda\big(\Par^{^{n_k}}(a)\to b_{(a,n_k)}\big)}\xrightarrow{k\to+\infty}0, \quad \forall a\in F_{2}.   
\end{equation}
Of course, for any $a\in F_2$, the cardinality of the set $\{b_{(a,n_k)}: k\in\NN\}$ is infinite. Passing to a subsequence, if needed, we can assume that $b_{(a,n_k)}\notin F$ for all $a\in F_2$.

For each $n\in \ZZ$, set $I_n:=F_2\cap \Gen_n$, where $\Gen_n$ is the $n$-th generation of $A$ with respect to some fixed vertex. Since $F_2$ is finite, there exist $s\in\NN$ and $\eta_1\leq \cdots\leq \eta_s$ in $\ZZ$ such that $I_{\eta_j}\neq\varnothing$ and
\[F_2=\bigcup_{j=1}^{s}I_{\eta_j}.\]
For any $k\in\NN$ and $1\leq j\leq s$, there exists some integer $t_{(k,j)}\geq1$ such that the set $I_{\eta_j}$ can be written as a $t_{(k,j)}$ finite union of disjoints sets $J_{(k,j,l)}$ such that two vertices $u,v$ belongs to  $J_{(k,j,l)}$ if $b_{(u,n_k)}=b_{(v,n_k)}$. We have then, for any $k\in\NN$ and $1\leq j\leq s$,
\[I_{\eta_j}=\bigcup_{l=1}^{t_{(k,j)}}J_{(k,j,l)}.\]
For any $k\in\NN$, $1\leq j\leq s$ and $1\leq l\leq t_{(k,j)}$, let us fix some vertex $c_{(k,j,l)}\in J_{(k,j,l)}$.

Now, for all $k$ sufficiently big such that $\Chi^{n_k}(F)\cap F=\varnothing$, we define 
\[h_{k}=f+\sum_{a\in F}\frac{g(v)^{1/m}}{\lambda(a\to u_{(a,n_k)})^{1/m}}e_{u_{(a,n_k)}} +e^{i \frac{\pi}{m}}\sum_{j=1}^{s}\sum_{l=1}^{t_{(k,j)}} \bigg(\sum_{v\in J_{(k,j,l)}}\dfrac{f(v)^m \lambda(\Par^{n_k}(v)\to v)}{\lambda(\Par^{n_k}(v)\to b_{(v,n_k)})}\bigg)^{1/m}e_{b_{\big(c_{(k,j,l)},n_k\big)}}.\]
Note that the terms of $h_k$ have disjoint supports. By using \eqref{lp:unrooted:eq1} and \eqref{F_j^2:unrooted:eq}, we have
\[h_k\xrightarrow{k\to+\infty}f\in U.\]
Thus we have $h_k\in U$ for all $k$ big enough. On the other hand, we have
\begin{align*}
    h_k^m = f^m+\sum_{a\in F}\frac{g(v)}{\lambda(a\to u_{(a,n_k)})}e_{u_{(a,n_k)}} -\sum_{j=1}^{s}\sum_{l=1}^{t_{(k,j)}} \bigg(\sum_{v\in J_{(k,j,l)}} \dfrac{f(v)^m\lambda(\Par^{n_k}(v)\to v)}{\lambda(\Par^{n_k}(v)\to b_{(v,n_k)})}\bigg)e_{b_{\big(c_{(k,j,l)},n_k\big)}}.
\end{align*}
Hence,
\begin{align*}
    &B_\lambda^{n_k}h_k^m
    =\sum_{a\in F}f^m(a) \lambda(\Par^{n_k}(a)\to a)e_{\Par^{n_k}(a)}+g\\
    &-\sum_{j=1}^{s}\sum_{l=1}^{t_{(k,j)}} \bigg(\sum_{v\in J_{(k,j,l)}} \dfrac{f(v)^m\lambda(\Par^{n_k}(v)\to v)}{\lambda(\Par^{n_k}(v)\to b_{(v,n_k)})}\bigg)\lambda\Big(\Par^{n_k}\big(b_{(c_{(k,j,l)},n_k)}\big)\to b_{\big(c_{(k,j,l)},n_k\big)}\Big) e_{\Par^{n_k}\big(b_{(c_{(k,j,l)},n_k)}\big)},
\end{align*}
since $c_{(k,j,l)}\in J_{(k,j,l)}$, for every $v\in J_{(k,j,l)}$, we have $b_{(v,n_k)}=b_{\big(c_{(k,j,l)},n_k\big)}$, hence
\begin{align*}
    B_\lambda^{n_k}h_k^m
    &=\sum_{a\in F}f^m(a) \lambda(\Par^{n_k}(a)\to a)e_{\Par^{n_k}(a)}+g\\
    &-\sum_{j=1}^{s}\sum_{l=1}^{t_{(k,j)}}\sum_{v\in J_{(k,j,l)}}\dfrac{f(v)^m \lambda(\Par^{n_k}(v)\to v)}{\lambda(\Par^{n_k}(v)\to b_{(v,n_k)})}\,\lambda\Big(\Par^{n_k}\big(b_{(v,n_k)}\big)\to b_{(v,n_k)}\Big) e_{\Par^{n_k}\big(b_{(v,n_k)}\big)}\\
    &=\sum_{a\in F}f^m(a) \lambda(\Par^{n_k}(a)\to a)e_{\Par^{n_k}(a)}+g\\
    &-\sum_{j=1}^{s}\sum_{l=1}^{t_{(k,j)}} \sum_{v\in J_{(k,j,l)}}\dfrac{f(v)^m \lambda(\Par^{n_k}(v)\to v)}{\lambda(\Par^{n_k}(v)\to b_{(v,n_k)})}\,\lambda\Big(\Par^{n_k}(v)\to b_{(v,n_k)}\Big) e_{\Par^{n_k}(v)}\\
     &=\sum_{a\in F}f^m(a) \lambda(\Par^{n_k}(a)\to a)e_{\Par^{n_k}(a)}+g-\sum_{j=1}^{s}\sum_{l=1}^{t_{(k,j)}} \sum_{v\in J_{(k,j,l)}}f(v)^m \lambda(\Par^{n_k}(v)\to v)  e_{\Par^{n_k}(v)}\\
     &=\sum_{a\in F}f^m(a) \lambda(\Par^{n_k}(a)\to a)e_{\Par^{n_k}(a)}+g-\sum_{v\in F_2}f(v)^m \lambda(\Par^{n_k}(v)\to v)  e_{\Par^{n_k}(v)}\\
      &=\sum_{a\in F_1}f^m(a) \lambda(\Par^{n_k}(a)\to a)e_{\Par^{n_k}(a)}+g,
\end{align*}
now, by using \eqref{F_j^1:unrooted:eq}, we get
\[
B_\lambda^{n_k}h_k^m\xrightarrow{k\to+\infty}g\in V.
\]
Hence, we have $B_\lambda^{n_k}h_k^m\in V$ for all $k$ big enough. Therefore, the conclusion follows from Proposition \ref{prop:hc:powers}.
\end{proof}

Specially in the case of $\ell^1$-spaces of a tree, condition (v) in Theorem \ref{thm:necessary:unrooted} is equivalent to the hypercyclicity of $B_\lambda$ (see \cite[Theorem 5.3]{Karl1}). This gives the following readable corollary.
\begin{corollary}
    Let $A$ be an unrooted directed tree and $\lambda=(\lambda_v)_{v\in A}$ be a weight such that its induced backward shift $B_\lambda:\ell^1(A)\to \ell^1(A)$ is bounded. Then $B_\lambda$ is hypercyclic if, and only if, there is a residual set of elements $f$ such that $f^m\in HC(B_\lambda)$ for all $m\geq1$.
\end{corollary}
In particular, if one wants to prove that a hypercyclic weighted backward shift on $\ell^1(A)$ does not admit a hypercyclic algebra, then a different method than the one used by Aron et al. in \cite{aron2007powers} is required. Indeed, for proving that translations don't admit hypercyclic algebras, they have shown that the these operators admit no vector with hypercyclic powers. Alas, hypercyclic shifts always have many hypercyclic powers. A similar fact also happens with non-trivial convolution operators $\phi(D)$ with $|\phi(0)|>1$. We don't know exactly which ones admit hypercyclic algebras, but Bayart proved in \cite[Theorem 1.3]{bayart2019algebra} that they do have many hypercyclic powers.

It was shown in \cite[Corollary 5.4]{Karl1} that, on any unrooted directed tree $A$, a Rolewicz operators on $\ell^1(A)$ is never hypercyclic. However, Rolewicz operators on $\ell^p$-spaces for $1 < p < \infty$ can be hypercyclic (even mixing and even chaotic, see \cite[Theorem 9.4]{Karl2}).
By using the previous theorem, we can deduce the following consequence. 

\begin{corollary}\label{corol:rolewicz}
    Let $A$ be an unrooted directed tree and $\lambda\in \CC$. Then no Rolewicz operator $\lambda B$, acting on $\ell^p(A)$ with $1 \leq p < \infty$, supports a dense hypercyclic algebra. 
\end{corollary}
We have chosen to state this corollary this way for simplicity, but the conclusion is actually stronger: these operators admit no dense set of vectors with hypercyclic powers. In particular, Theorem \ref{thm:generalcriterion} does not apply to them. Therefore, if they do admit hypercyclic algebras, then these cannot come from a Baire argument.

In stark contrast with Theorem \ref{thm:necessary:unrooted}, we do not know if (v) $\Rightarrow$ (i) or (ii) in the previous theorem. More precisely, if we look into the second condition of Theorem \ref{thm:necessary:unrooted}(v), then we have two possibilities for each $v\in A$: either
\begin{equation}\label{eq:poss=bad}
    \sup_{u\in \Chi^{n_k}(\Par^{n_k}(v))}\frac{\big|\lambda(\Par^{n_k}(u)\to u)\big|}{|\lambda(\Par^{n_k}(v)\to v)\big|}\xrightarrow[]{k\to+\infty} +\infty,
\end{equation}
or
\begin{equation}\label{eq:poss=good}
    \lambda(\Par^{n_k}(v)\to v)\xrightarrow[]{k\to+\infty} 0.
\end{equation}
We do not know how to produce hypercyclic algebras when there are $v\in A$ satisfying \eqref{eq:poss=bad} and not \eqref{eq:poss=good}. On the other hand, when \eqref{eq:poss=good} happens for all $v\in A$, then it is not difficult to adapt to the unrooted case our results from Section \ref{sec:rooted}.

\begin{theorem}\label{xico}
    Let $A$ be an unrooted directed tree and $\lambda=(\lambda_v)_{v\in A}$ be a weight such that its induced backward shift $B_\lambda:\ell^p(A)\to \ell^p(A)$ is bounded, with $1\leq p<\infty$. Suppose that there is a sequence of non-negative integers $(n_k)_{k\geq1}$ such that, for every $v\in A$,
    \begin{align*}
         \sup_{u\in \Chi^{n_k}(v)}|\lambda(v\to u)| \xrightarrow{k\to+\infty} +\infty \quad\text{ and }\quad\lambda(\Par^{n_k}(v)\to v)\xrightarrow[]{k\to+\infty} 0.
    \end{align*}
    Then $B_\lambda$ has a dense, countably generated, free hypercyclic algebra.
\end{theorem}
\begin{proof}
Let $d\geq 1$, let $P\subset \NN_0^d\backslash\{(0,...,0)\}$ be finite and let $U_1,U_2,...,U_d, V,W\subset \ell^p(A)$ be non-empty and open, with $0\in W$. We fix $(f_1,\dots, f_d)\in U_1\times \cdots\times U_d$ and $g\in V$ with support on some finite subset $F$ of $A$. From Lemma \ref{lemma:maps}, there are $s\in (0,+\infty)^d$ and $\beta\in P$ such that the maps  $L_\alpha:\RR^d\to\RR$ defined $L_\alpha(s)=\sum_{j=1}^d\alpha_js_j$, for $s\in\RR^d$ and $\alpha\in P$, satisfy $1=L_\beta(s)<L_\alpha(s)$ for all $\alpha\in P\backslash\{\beta\}$.

Now, from the hypothesis, there exists an increasing sequence of positive entire numbers $(n_k)_{k\geq1}$ such that, for all $v\in A$,
    \[C_k(v):=\sup_{u\in \Chi^{n_k}(v)}\big|\lambda(v\to u)\big|\xrightarrow{k\to\infty} +\infty.\]
Thus, for each $a\in A$ and each $k\in\NN$, we can find $u_{a,k}\in \Chi^{n_k}(a)$ such that
\begin{equation}\label{arafat-rep}
    \frac{1}{|\lambda(a\to u_{a,k})|}< \frac{1}{C_k(a)}+\frac{1}{k}.
\end{equation}
We define $g_{a,k}:A\to \KK$ by \[g_{a,k}=\frac{1}{\lambda(a\to u_{a,k})}e_{u_{a,k}}\]
and we notice that $(B_\lambda^{n_k}e_{u_{a,k}})(v)=\lambda(a\to u_{a,k})e_{a}(v)$. Now, for all $k$ sufficiently big (conditions on the size of $k$ will be given during the proof), we set $h_k=(h_{k,1},\dots,h_{k,d})$, where
\[h_{k,j}=f_j+\sum_{a\in F}\big(g(a)g_{a,k}\big)^{s_j},\quad j=1,\dots, d.\]
If $k\in\NN$ is big enough, we will have $\Chi^{n_k}(F)\cap F=\varnothing $. Thus, for each $\alpha\in P$,
\begin{align*}
    h_k^\alpha = f_1^{\alpha_1}\cdots f_d^{\alpha_d}+\sum_{a\in F} \big(g(a)g_{a,k}\big)^{L_\alpha(s)}.
\end{align*}
Hence,
\begin{align*}
    B_\lambda^{n_k}h_k^\alpha
    &=\sum_{a\in F}(f_1^{\alpha_1}\cdots f_d^{\alpha_d})(a) \lambda(\Par^{n_k}(a)\to a)e_{\Par^{n_k}(a)}+\sum_{a\in F}\frac{g(a)^{L_\alpha(s)}}{\lambda(a\to u_{a,k})^{L_\alpha(s)-1}} e_{a}.
\end{align*}
If $\alpha=\beta$, we have 
\[B_\lambda^{n_k}h_k^\beta=\sum_{a\in F}(f_1^{\beta_1}\cdots f_d^{\beta_d})(a) \lambda(\Par^{n_k}(a)\to a)e_{\Par^{n_k}(a)}+\sum_{a\in F}g(a)e_a\xrightarrow{k\to+\infty}g\in V;\]
and if $\alpha\in P\backslash\{\beta\}$, since
\begin{equation*}
    \Big\|B_\lambda^{n_k}h_k^\alpha\Big\|_{p}^{p}
    = \sum_{a\in F}|(f_1^{\alpha_1}\cdots f_d^{\alpha_d})(a)|^p |\lambda(\Par^{n_k}(a)\to a)|^p+    
    \sum_{a\in F} \frac{|g(a)|^{pL_\alpha(s)}}{|\lambda(a\to u_{a,k})|^{p(L_\alpha(s)-1)}},
\end{equation*}
and since $L_\alpha(s)-1>0$, from \eqref{arafat-rep}, we find
\[\Big\|B_\lambda^{n_k}h_k^\alpha\Big\|_{p}^{p}\xrightarrow{k\to+\infty}0.\]
Thus, $B_\lambda^{n_k}h_k^\alpha\in W$ if $k$ is big enough. Finally, still from \eqref{arafat}, we get
\begin{align*}
    \|h_{k,j}-f_j\|_{p}^{p}&\leq
    \sum_{a\in F}|g(a)|^{ps_j}\frac{1}{|\lambda(a\to u_{a,k})|^{ps_j}}\xrightarrow{k\to+\infty} 0.
\end{align*}
Hence, if $k$ is big enough, $h_{k,j}\in U_j$ for all $j=1,\dots,d$. This completes the proof.
\end{proof}

Notice that, for symmetric weights, condition \eqref{eq:poss=bad} never happens, for if this is the case, then all quotients inside the supremum are ones. Therefore, we can combine Theorem \ref{thm:necessary:unrooted} with Theorem \ref{xico} to get the following satisfying consequence.
\begin{corollary}\label{corol:symmetric:carac}
    Let $A$ be an unrooted directed tree, $\lambda=(\lambda_v)_{v\in A}$ be a symmetric weight and $1\leq p<\infty$. 
Suppose that $B_\lambda$ is bounded on $\ell^p(A)$. The following are equivalent.
\begin{enumerate}[$(i)$]
    \item $B_\lambda$ supports a dense, countably generated, free hypercyclic algebra.
    \item $B_\lambda$ supports a dense hypercyclic algebra. 
    \item There is a dense set of elements $f\in\ell^p(A)$ such that $f^m\in  HC(B_\lambda)$ for some integer $m \geq p$.
    \item There is a sequence of non-negative integers $(n_k)_{k\geq1}$ such that, for every $v\in A$,
    \[\sup_{u\in \Chi^{n_k}(v)}|\lambda(v\to u)| \xrightarrow{k\to+\infty} +\infty \quad\text{ and }\quad\lambda(\Par^{n_k}(v)\to v)\xrightarrow[]{k\to+\infty} 0.\]
\end{enumerate}
\end{corollary}

Notice that \eqref{eq:poss=good} is always the case when the tree in question has a free left end, for in this case the quotients inside the supremum in \eqref{eq:poss=bad} are constant for all $k$ sufficiently big. This lead immediately to the following nice complement of Theorem \ref{cara:lp:hc-alg-rooted} for the case of unrooted trees with a free left end.

\begin{corollary}\label{corol:new}
Let $A$ be an unrooted directed tree with a free left end, let $1\leq p<+\infty$ and let $\lambda=(\lambda_v)_{v\in A}$ be a weight such that its induced backward shift $B_\lambda:\ell^p(A)\to \ell^p(A)$ is bounded. The following are equivalent. 
\begin{enumerate}[$(i)$]
    \item $B_\lambda$ supports a dense, countably generated, free hypercyclic algebra.
    \item $B_\lambda$ supports a dense hypercyclic algebra.
    \item There is a dense set of vectors generating a hypercyclic algebra for $B_\lambda$.
    \item There is a dense set of elements $f\in \ell^p(A)$ such that $f^m\in HC(B_\lambda)$ for some integer $m\geq p$.
    \item There is a sequence of positive integers $(n_k)_{k\geq1}$ such that, for every $v\in A$, 
     \[\sup_{u\in \Chi^{n_k}(v)}|\lambda(v\to u)| \xrightarrow{k\to+\infty} +\infty \quad\text{ and }\quad\lambda(\Par^{n_k}(v)\to v)\xrightarrow[]{k\to+\infty} 0.\]
\end{enumerate}
\end{corollary}

Let us now turn our attention to the case of $c_0$-spaces. In the following result, we obtain analogous implications to the the ones found in Theorem \ref{thm:necessary:unrooted}. Here, as before, property \eqref{key-lemma} is going to be quite useful.

\begin{theorem}\label{thm:necessary:c0}
    Let $A$ be an unrooted directed tree and $\lambda=(\lambda_v)_{v\in A}$ be a weight such that its induced backward shift $B_\lambda:c_0(A)\to c_0(A)$ is bounded. The following assertions are equivalent. 
    \begin{enumerate}[$(i)$]
        \item $B_\lambda$ is hypercyclic. 
        \item There is a residual set of elements $f\in c_0(A)$ such that $f^m\in HC(B_\lambda)$ for all $m\geq 1$.
        \item There is a sequence of non-negative integers $(n_k)_{k\geq1}$ such that, for every $v\in A$,
\begin{equation}\label{hyp:c0:1}
    \sum_{u\in \Chi^{n_k}(v)}\big|\lambda(v\to u)\big| \xrightarrow{k\to+\infty}+\infty,
\end{equation}
and
\begin{equation}\label{hyp:c0:2}
    \max\Bigg(\dfrac{1}{|\lambda(\Par^{n_k}(v)\to v)|},\sum_{u\in \Chi^{n_k}(\Par^{n_k}(v))}\dfrac{\big|\lambda(\Par^{n_k}(v)\to u)\big|}{|\lambda(\Par^{n_k}(v)\to v)|}\Bigg) \xrightarrow{k\to+\infty}+\infty.
\end{equation}
    \end{enumerate}
\end{theorem}
\begin{proof}
    Is is clear that $(ii)\Rightarrow (i)$. Moreover, by \cite[Theorem 5.3]{Karl1}, we have $(i)\Leftrightarrow (iii)$. Let us show that  $(iii)\Rightarrow (ii)$ by using Proposition \ref{prop:hc:powers}. Let  $U, V$ be non-empty open subsets of $c_0(A)$ and let $m \geq 1$. We fix $f\in U$ and $g \in V$ with support on some finite subset $F$ of $A$. 

    From the hypothesis, there is an increasing sequence of positive integers  $(n_k)_{k\geq1}$ such that \eqref{hyp:c0:1} and \eqref{hyp:c0:2} are satisfied for all $v\in A$.
    By extracting a subsequence if needed, we can write $F= F_{1} \cup F_{2}$, with $F_{1} \cap F_{2} = \varnothing$, such that 
     \begin{equation}\label{c0:unrooted:eq1}
      \sum_{u\in \Chi^{n_k}(v)}\big|\lambda(v\to u)\big|\xrightarrow{k\to+\infty}+\infty,\quad \forall v\in F,
    \end{equation}
    \begin{equation}\label{c0:F_j^1:unrooted:eq}
    \lambda\big(\Par^{n_{k}}(v)\to v\big)\xrightarrow{k\to+\infty}0, \quad \forall v\in F_{1},
\end{equation}
     and
\begin{equation}\label{c0:F_j^2:unrooted:eq}
     \sum_{u\in \Chi^{n_k}(\Par^{n_k}(v))\setminus F}\dfrac{\big|\lambda(\Par^{n_k}(v)\to u)\big|}{|\lambda(\Par^{n_k}(v)\to v)|}\xrightarrow{k\to+\infty} +\infty, \quad \forall v\in F_{2}.   
\end{equation}
For each $v\in F$ and $k\in \NN$, by applying \eqref{key-lemma} with $J=\Chi^{n_k}(v)$ and using \eqref{c0:unrooted:eq1},  we get the existence of $g_{v,k}:A\to \KK$ with finite support in $\Chi^{n_k}(v)$ satisfying
\begin{align}\label{Figue:eq1}
\|g_{v,k}\|_1=1\quad\text{and}\quad\sup_{u\in\Chi^{n_k}(v)}\frac{|g_{v,k}(u)|}{|\lambda(v\to u)|}\xrightarrow{k\to+\infty}0.
\end{align}
For each $v\in F$ and $k\in\NN$, we define 
\[R_{v,k} = \sum_{u\in\Chi^{n_k}(v)} \frac{|g_{v,k}(u)|}{\lambda(v\to u)}e_u\]
and we notice that $R_{v,k}\in c_0(A)$ and $B_{\lambda}^{n_k}R_{v,k}=e_v$. Also, by using \eqref{Figue:eq1} we get that $\|R_{v,k}^{1/m}\|_{\infty}\xrightarrow{k\to+\infty}0$.
Moreover, for each $v\in F_2$ and $k\in \NN$, by applying again  \eqref{key-lemma} with $J=\Chi^{n_k}(\Par^{n_k}(v))\setminus F$ and using \eqref{c0:F_j^2:unrooted:eq},  we get the existence of $h_{v,k}:A\to \KK$ with finite support in $\Chi^{n_k}(\Par^{n_k}(v))\setminus F$ satisfying
\begin{align}\label{Figue:eq2}
\|h_{v,k}\|_1=1\quad\text{and}\quad\sup_{u\in\Chi^{n_k}(\Par^{n_k}(v))}\dfrac{|h_{v,k}(u)|\,|\lambda(\Par^{n_k}(v)\to v)|}{\big|\lambda(\Par^{n_k}(v)\to u)\big|}\xrightarrow{k\to+\infty}0.
\end{align}
Let us define, for every $v\in F_2$
\[H_{v,k}=\sum_{u\in \Chi^{n_k}(\Par^{n_k}(v))\setminus F}\frac{\lambda(\Par^{n_k}(v)\to v)}{\lambda(\Par^{n_k}(v)\to u)}|h_{v,k}(u)| e_u.\]
It follows that $H_{v,k}\in c_0(A)$,  $B_{\lambda}^{n_k}H_{v,k}=\lambda(\Par^{n_k}(v)\to v) e_{\Par^{n_k}(v)}$ and $\|H_{v,k}\|_{\infty}\xrightarrow{k\to+\infty}0$.

We define a subset $G \subset F_2$ by picking exactly one vertex from $F_2$ per generation. We can assume that $G=\{v_1,\ldots,v_s\}$, for some $s\geq1$. Let $M\geq1$ be big enough so that, for any $j\in\{1,\dots,s\}$ and $v\in F_2$ such that $v\sim v_j$, we have $\Par^M(v_j)=\Par^M(v)$.
For $k$ big enough such that $\Chi^{n_k}(F)\cap F=\varnothing$, we define
\begin{align*}
  h_k &= f+\sum_{v\in F}g(v)^{1/m}R_{v,k}^{1/m}+e^{i\frac{\pi}{m}}\sum_{j=1}^{s}\Bigg(\sum_{ \underset{v\sim v_j}{v\in F_2}}f(v)^m\dfrac{\lambda(\Par^M(v)\to v)}{\lambda(\Par^M(v_j)\to v_j)}\Bigg)^{1/m} H_{v_j,k}^{1/m}.
\end{align*}
Note that $h_{k} \xrightarrow{k\to +\infty} f\in U$, thus $h_{k}\in U$ for all $k$ big enough. Moreover, for $k$ big enough such that $n_k>M$, we have
\begin{align*}
    B_{\lambda}^{n_k}h_{k}^{m}
    &=\sum_{v\in F}f^m(v)\lambda(\Par^{n_k}(v)\to v)e_{\Par^{n_k}(v)} + g \\
    &\quad \quad -\sum_{j=1}^{s}\!\Bigg(\!\sum_{ \underset{v\sim v_j}{v\in F_2}}f(v)^m\dfrac{\lambda(\Par^M(v)\to v)}{\lambda(\Par^M(v_j)\to v_j)}\!\Bigg)  \lambda(\Par^{n_k}(v_j)\to v_j) e_{\Par^{n_k}(v_j)}\\
    &=\sum_{v\in F}f^m(v)\lambda(\Par^{n_k}(v)\to v)e_{\Par^{n_k}(v)}+g-\sum_{ v\in F_2 }f(v)^m \lambda(\Par^{n_k}(v)\to v) e_{\Par^{n_k}(v)}   \\
    &=\sum_{v\in F_1}f^m(v)\lambda(\Par^{n_k}(v)\to v)e_{\Par^{n_k}(v)}+g.
\end{align*}
By using \eqref{c0:F_j^1:unrooted:eq}, we get $B_{\lambda}^{n_k}h_{k}^{m} \xrightarrow{k\to +\infty} g\in V$. Thus, $B_{\lambda}^{n_k}h_{k}^{m} \in V$ for all $k$ big enough. Therefore, the result follows from Proposition \ref{prop:hc:powers}.
\end{proof}

Just like for $\ell^p$-spaces, a stronger assumption as \eqref{eq:poss=good} on $c_0$-spaces allows us to adapt to the unrooted case what we managed to do on rooted trees.

\begin{theorem}\label{hcalg:c0:unrooted}
    Let $A$ be an unrooted directed and $\lambda=(\lambda_v)_{v\in A}$ be a weight such that its induced backward shift $B_\lambda:c_0(A)\to c_0(A)$ is bounded. Suppose that there is a sequence of non-negative integers $(n_k)_{k\geq1}$ such that, for every $v\in A$,
    \begin{equation*}
        \sum_{u\in \Chi^{n_k}(v)}|\lambda(v\to u)| \xrightarrow{k\to+\infty} +\infty.\quad\text{and}\quad \lambda(\Par^{n_k}(v)\to v)\xrightarrow{k\to+\infty}0.
    \end{equation*}
    Then $B_\lambda$ has a dense, countably generated, free hypercyclic algebra.
\end{theorem}
\begin{proof}
Let $d\geq 1$, let $P\subset \NN_0^d\backslash\{(0,...,0)\}$ finite and let $U_1,U_2,...,U_d,V,W\subset c_0(A)$ non-empty and open subsets with $0\in W$. We fix $(f_1,\dots, f_d)\in U_1\times \cdots\times U_d$ and $g\in V$ with support on some finite subset $F$ of $A$. Just like before, we apply Lemma \ref{lemma:maps} and we find $s\in (0,+\infty)^d$ and $\beta\in P$ such that the maps $L_\alpha:\RR^d\to\RR$ given $L_\alpha(s)=\sum_{j=1}^d\alpha_js_j$, for $s\in\RR^d$ and $\alpha\in P$, satisfy $1=L_\beta(s)<L_\alpha(s)$ for all $\alpha\in P\backslash\{\beta\}$.

From the hypothesis, there is $(n_k)_{k\geq1}$ an increasing sequence of positive entire numbers such that, for all $v\in A$,
\[\sum_{u\in \Chi^{n_k}(v)}\big|\lambda(v\to u)\big|\xrightarrow{k\to+\infty}+\infty.\]
For each $v\in A$ and $k\in \NN$, we apply \eqref{key-lemma} with $J=\Chi^{n_k}(v)$ and we get the existence of $g_{v,k}:A\to \KK$ with support in $\Chi^{n_k}(v)$ satisfying
\begin{align}\label{champignon-rep}
\|g_{v,k}\|_1=1\quad\text{and}\quad\sup_{u\in\Chi^{n_k}(v)}\frac{|g_{v,k}(u)|}{|\lambda(v\to u)|}\xrightarrow{k\to+\infty}0.
\end{align}
For each $v\in A$ and $k\in\NN$, we define \[R_{v,k} = \sum_{u\in\Chi^{n_k}(v)} \frac{|g_{v,k}(u)|}{\lambda(v\to u)}e_u\]
and we notice that $B_{\lambda}^{n_k}R_{v,k}=e_v$.
Hence, for $k\in\NN$ big enough so that 
$\Chi^{n_k}(F)\cap F=\varnothing $,  we set $h_k=(h_{k,1},\dots,h_{k,d})$, where  
\[h_{k,j}=f_j+\sum_{v\in F}g(v)^{s_j}R_{v,k}^{s_j}, \quad j=1,\dots,d,\]
and we get
\begin{align*}
    B_{\lambda}^{n_k}h_k^{\beta}= \sum_{v\in F} (f_1^{\beta_1}\cdots f_d^{\beta_d})(v) \lambda(\Par^{n_k}(v)\to v)e_{\Par^{n_k}(v)}+g\xrightarrow{k\to+\infty}g\in V.
\end{align*}
We also have, for each $v\in F$ and $j=1,\dots, d$,
\begin{align*}
    \|R_{v,k}^{s_j}\|_{\infty} 
    &= \bigg(\sup_{u\in \Chi^{n_k}(v)}\frac{|g_{v,k}(u)|}{|\lambda(v\to u)|}\bigg)^{s_j}\xrightarrow{k\to +\infty}0,
\end{align*}
where the last limit comes from \eqref{champignon-rep}. Thus, $h_{k,j}\in U_j$, for $j=1,\ldots,d$, if $k$ is big enough. Now, we consider $\alpha \in P\backslash\{\beta\}$ and $k\in\NN$, and we calculate
\begin{align*}
\|B_{\lambda}^{n_k}h_k^{\alpha}\|_{\infty}
&= \bigg\|\sum_{v\in F} (f_1^{\alpha_1}\cdots f_d^{\alpha_d})(v) \lambda(\Par^{n_k}(v)\to v)e_{\Par^{n_k}(v)}+\sum_{v\in F}g(v)^{L_\alpha(s)}B_\lambda^{n_k} R_{v,k}^{L_\alpha(s)}\bigg\|_{\infty}\\
&=\max_{v\in F}\bigg\{\big|(f_1^{\alpha_1}\cdots f_d^{\alpha_d})(v) \lambda(\Par^{n_k}(v)\to v)\big|,\bigg|\sum_{u\in \Chi^{n_k}(v)}g(v)^{L_\alpha(s)}\frac{|g_{v,k}(u)|^{L_\alpha(s)}}{\lambda(v\to u)^{L_\alpha(s)-1}}\bigg|\bigg\}\\
&\leq \max_{v\in F}\bigg\{\|f_1^{\alpha_1}\cdots f_d^{\alpha_d}\|_\infty|\lambda(\Par^{n_k}(v)\to v)|+\|g^{L_\alpha(s)}\|_\infty\sup_{u\in\Chi^{n_k}(v)}\bigg(\frac{|g_{v,k}(u)|}{|\lambda(v\to u)|}\bigg)^{L_\alpha(s)-1}\bigg\},\\
\end{align*}
where the last term vanishes as $k\to+\infty$ as a consequence of \eqref{champignon-rep} and the hypothesis. Therefore, if $k$ is big enough, we have $B_{\lambda}^{n_k}h_k^{\alpha}\in W$ for all $\alpha\in P\backslash\{\beta\}$. From Theorem \ref{thm:generalcriterion}, $B_\lambda$ has a dense, countably generated, free hypercyclic algebra on $c_0(A)$.
\end{proof}

In both cases of $c_0$-spaces and $\ell^1$-spaces, if the weight $\lambda$ satisfies \eqref{eq:poss=good} for the whole sequence $(n)_{n\in\NN}$, then we get the following equivalence.

\begin{corollary}
Let $A$ be an unrooted directed, let $X=\ell^1(A)$ or $X=c_0(A)$ and let $\lambda=(\lambda_v)_{v\in A}$ be a weight such that its induced backward shift $B_\lambda:X\to X$  is bounded. Suppose that, for all $v\in A$, 
\[\lim_{n\to+\infty} \lambda(\Par^n(v)\to v)=0.\] 
Then $B_\lambda$ is hypercyclic if and only if it has a hypercyclic algebra.
\end{corollary}

Let us now provide a combined corollary using our previous results. 
From \cite[Corollary 5.6(b)]{Karl1} it follows that, for trees with a free left and, condition $(v)$ in Theorem \ref{thm:necessary:unrooted} characterizes when $B_\lambda$ is hypercyclic on $\ell^1(A)$. Similarly, also from \cite[Corollary 5.6(b)]{Karl1} it follows that, for unrooted trees with a free left end, Theorem \ref{hcalg:c0:unrooted} characterizes when $B_\lambda$ is hypercyclic on $c_0(A)$. Therefore, we get immediately the following.

\begin{corollary}
Let $A$ be an unrooted directed tree with a free left end, let $X=\ell^1(A)$ or $X=c_0(A)$ and let $\lambda=(\lambda_v)_{v\in A}$ be a weight such that its induced backward shift $B_\lambda:X\to X$ is bounded. Then the following are equivalent.
\begin{enumerate}[$(i)$]
    \item $B_\lambda$ has a dense, countably generated, free hypercyclic algebra.
    \item $B_\lambda$ has a hypercyclic algebra.
    \item $B_\lambda$ is hypercyclic.
    \item There is an increasing sequence of positive integers $(n_k)_{k\geq1}$ such that, for every $v\in A$, 
    \[\lambda(\Par^{n_k}(v)\to v)\xrightarrow[]{k\to+\infty} 0\]
    and
    \begin{align*}
        \sup_{u\in \Chi^{n_k}(v)}|\lambda(v\to u)| \xrightarrow{k\to+\infty} +\infty, \text{ if } X&=\ell^1(A),\\
        \sum_{u\in \Chi^{n_k}(v)}|\lambda(v\to u)| \xrightarrow{k\to+\infty} +\infty, \text{ if } X&=c_0(A).
    \end{align*}
\end{enumerate}
\end{corollary}

As a simple example, we can prove that the hypercyclic bilateral generalization of the Rolewicz operator has a hypercyclic algebra whenever the tree does not have too many branches.

\begin{example}
 Let $A$ be an unrooted directed tree with a free left end (a particular case being when $A$ has a bounded number of branches) and let $X=c_0(A)$ or $X=\ell^p(A), 1\leq p<+\infty$. Fixing $\lambda$ with $|\lambda|>1$, define the weights $(\lambda_v)_v$ as follows. Fix $n', n''\in \ZZ$ and define
    \[\begin{cases}
        \lambda_v=\frac{1}{\lambda}&\text{for all }   v\in\Gen_k\text{ with }k<n',\\
        \lambda_v=\lambda&\text{for all }v\in\Gen_k \text{ with }k>n'',\\
        \lambda_v\neq 0 &\text{for all }v\in \Gen_k\text{ with }k\in[\![n',n'']\!].
    \end{cases}\]
    Then, if $B_\lambda:X\to X$ is bounded, it has a hypercyclic algebra.
\end{example}

\section{Existence of hypercyclic algebras on trees}\label{sec:existence}

It is known from \cite[Theorem 6.1]{Karl1} that any $\ell^p$-space, $1\leq p<+\infty$, and $c_0$-space of a leafless directed tree (rooted or unrooted) supports a mixing, as well as a hypercyclic non-mixing, weighted backward shift. These ideas can be easily adjusted to fit in the context of hypercyclic algebras. We state it here as a theorem, although this is more of a remark. We can prove the following.

\begin{theorem}
    Let $V$ be a leafless directed tree. Let $X=\ell^p(V)$, $1\leq p <\infty$, or $X=c_0(V)$.
    \begin{enumerate}[$(i)$]
        \item There exists a mixing weighted backward shift on $X$ that supports  a dense, countably generated, free hypercyclic algebra.
        \item There exists a non-mixing weighted backward shift on $X$ that supports a dense, countably generated, free hypercyclic algebra.
    \end{enumerate}
\end{theorem}
\begin{proof}
Note that, when $V$ is a rooted directed tree and $X=\ell^1(V)$ or $c_0(V)$, then it is enough to apply  Corollary \ref{hc_alg_l1_rooted}  or Theorem \ref{hc_alg_c0_rooted} combined with \cite[Theorem 6.1]{Karl1}.

We will prove the theorem for the space $X=\ell^p(V)$, where $1 < p < \infty$ and $V$ is an unrooted leafless directed tree. The  cases of $\ell^1(V)$ or $c_0(V)$ are simpler. The same technique works for rooted trees as well, we can obtain this by repeating the same steps for the $n$th generations, with $n \in \mathbb{N}_0$. We fix $v_0\in V$ and we enumerate the generations in $A$ with respect to $v_0$.

$(i).$ Define the exact same mixing weighted backward shift $B_\lambda$ as in \cite[Theorem 6.1 (a)]{Karl1}. Then it is easy to check that $B_\lambda$ supports a dense, countably generated, free hypercyclic algebra, for it satisfies the hypothesis of Theorem \ref{xico}. 
    
$(ii)$. Once more, we consider the same non-mixing weighted backward shift $B_\lambda$ as in \cite[Theorem 6.1 (b)]{Karl1}. It is not difficult to verify that it satisfies the hypothesis of Theorem \ref{xico}, which implies that it has a dense, countably generated, free hypercyclic algebra.
\end{proof}

In the following, we will address the question of the existence of a mixing backward shift on rooted directed trees that does not support a hypercyclic algebra. Note that, in the case of $\ell^1$-spaces or $c_0$-spaces, any mixing backward shift on a rooted directed tree supports a hypercyclic algebra.  We know from the examples we have that this is not true for $\ell^p$-spaces, with $p > 1$.

Consider $A$ to be a rooted directed tree. We aim to discuss the possibility of finding a weighted backward shift $B_\lambda$ on $\ell^p(A)$, with $p>1$, that is mixing but does not have a hypercyclic algebra. Our weight should satisfies the following conditions:
\begin{align}
&\sup_{v\in A}\sum_{u\in \Chi(v)}|\lambda_u|^{p^\ast}<\infty, \quad \text{(continuity)}\label{equ13}\\   
& \sum_{u\in \Chi^n(v)}|\lambda(v\to u)|^{p^\ast}\xrightarrow{n\to+\infty}+\infty, \ \text{for all } v\in A,\label{equ14} \quad \text{(mixing)}\\
&\limsup_{n\to+\infty}\!\sup_{u\in\Chi^n(v_0)}\!|\lambda(v_0\to u)|<+\infty, \text{ for some } v_0\in A\label{equ15} \quad \text{(no hypercyclic algebra).}
\end{align} 

Note that, if there exists a given $v_0\in A$ for which the conditions \eqref{equ14} and \eqref{equ15} hold, then
\begin{equation}\label{jambo}
|\Chi^n(v_0)|\xrightarrow{n\to+\infty}+\infty.
\end{equation}
Indeed, otherwise, there would exist $M>0$ such that $|\Chi^n(v_0)|<M$ for all $n\in\NN_0$. Therefore,
\[\sum_{u\in \Chi^n(v_0)}|\lambda(v_0\to u)|^{p^\ast}\leq M \sup_{u\in\Chi^n(v_0)}\,|\lambda(v_0\to u)|^{p^\ast},\]
by using \eqref{equ14}, we obtain
\[\sup_{u\in\Chi^n(v_0)}\,|\lambda(v_0\to u)|\xrightarrow{n\to+\infty}+\infty,\]
which contradicts \eqref{equ15}. 

We can now turn the previous discussion into an answer to the problem of finding a mixing backward shift with no hypercyclic algebra. We say that a vertex $v_0\in A$ is \emph{fertile} when $|\Chi^n(v)|\to+\infty$ as $n\to+\infty$ for all $v\in \bigcup_{n\in\NN}\Chi^n(v_0)$. Then we have the following.
\begin{theorem}\label{thm:fertile}
    Let $A$ be a rooted directed tree. Then $A$ admits a mixing weighted backward shift $B_\lambda:\ell^p(A)\to \ell^p(A)$, $1<p<\infty$, with no hypercyclic algebra if, and only if, $A$ has a fertile vertex.
\end{theorem}
\begin{proof}
    Suppose there exist a fertile vertex $v_0\in A$. We consider the sub-tree $B=\bigcup_{n\in \NN_0} \Chi^n(v_0)$ starting at $v_0$. We shall define $(\lambda_v)_{v\in A}$ first for $v\in B$, then for $v\in A\backslash B$. For $v\in B$, we define $\lambda_v\leq 1$ in a way that $B_\lambda :\ell^p(B)\to \ell^p(B)$ is continuous and mixing. To do this, let us start with $\lambda_{v_0}=1$. Now, if $\Chi(v_0)$ is infinite, we enumerate it as $\Chi(v_0)=\{u_l\}_{l\in\NN}$ and we define $\lambda_{u_1}=\lambda_{u_2}=1$ and $\lambda_{u_l}=\big(\frac{1}{2^{l-1}}\big)^{1/p\ast}$ for all $l\geq 2$. If $\Chi(v_0)$ has $m\geq 2$ elements, say $\Chi(v_0)=\{u_1,\ldots, u_m\}$, then we choose $\lambda_{u_1}=\lambda_{u_2}=1$ and $\lambda_{u_k}=\big(\frac{1}{m-2}\big)^{1/p\ast}$ for all $k=3,\dots, m$. If $\Chi(v_0)$ is unitary, then we take $\lambda_{u}=1$ for the only $u\in\Chi(v_0)$. Assuming $\lambda_v$ was defined for all $v\in\Chi^l(v_0)$, $l=1,\dots, k$, we apply the same definitions for all vertices $u\in \Chi^{k+1}(v_0)=\bigcup_{v\in \Chi^k(v_0)} \Chi(v)$. Then $\sup_{v\in V}\sum_{u\in \Chi(v)}|\lambda_u|^{p^\ast}\leq 3<\infty$, thus $B_\lambda$ is continuous. Also, for all $v\in B$, since $|\Chi^n(v)|\to+\infty$ as $n\to+\infty$, we know that $\lambda_{u_1}=\lambda_{u_2}=1$ happens infinitely many times in the definition of $(\lambda_v)_{v\in B}$ for vertices inside $\bigcup_{n\in\NN} \Chi^n(v)$, thus there are an infinite number of branches starting from $v$ and going to infinity such that all weights are 1, and this implies \eqref{equ14}. Notice also that \eqref{equ15} holds because $0<\lambda_v\leq 1$ for all $v\in B$.
    
    We now concentrate on $\lambda_v$ for $v\in A\backslash B$. If $\root=v_0$ then we are done, for in this case $A=B$ and $B_\lambda$ is the weighted backward shift defined above. If $\root\neq v_0$, we define $\lambda_v$ for $v\in A\backslash B$ following the same ideas as in \cite[Theorem 6.1(a)]{Karl1} to ensure that $B_\lambda$ is continuous and there are always infinite branches on which the weights are eventually all 2 (when $\Par(v_0)$ has at least two children), what ultimately implies \eqref{equ14}.

    Now, suppose that $A$ has no fertile vertex. Let $\lambda\in\KK^A$ be a weight such that $B_\lambda:\ell^p(A)\to\ell^p(A)$ is continuous and mixing. Then \eqref{equ14} holds for $\lambda$. For all $v_0\in A$, since $v_0$ is not fertile, there must exist $v\in A$ which has $v_0$ as ancestor and satisfies $|\Chi^n(v)|=1$ for all $n\in\NN$. Applying \eqref{equ14} for $v$ we conclude that $\sup_{u\in\Chi^n(v)}|\lambda(v\to u)|\to+\infty$ an $n\to+\infty$, and this implies $\sup_{u\in\Chi^n(v_0)}|\lambda(v_0\to u)|\to+\infty$. From Theorem \ref{cara:lp:hc-alg-rooted} we conclude that $B_\lambda$ has a hypercyclic algebra.
\end{proof}

Since not always mixing implies the existence of a hypercyclic algebra for shift on trees, one could think that having many branches would allow the construction of a counterexample. However, the existence of a fertile vertex is much more stronger than the mere existence of fast growing generations.

\begin{example}
    Let us define a rooted directed tree $A$ with ``many'' branches (that is, $\big(|\Chi^n(r)|\big)_n$ is an exponentially growing sequence) in which every mixing backward shift $B_\lambda:\ell^p(A)\to\ell^p(A)$ support a hypercyclic algebra. Let $\root$ be the root of $A$. We assign one child to $\root$: $\Chi(\root)=\{\root_1\}$. Next, we give $\root_1$ two children, $\Chi(\root_1)=\{s_1, \root_2\}$. From $s_1$ we generate a ``stationary'' branch: $\Chi(s_1)=\{s_{1,1}\}, \Chi(s_{1,1})=\{s_{1,2}\}$ and so on, that is, $\Chi(s_{1,n})=\{s_{1,n+1}\},$ for all $n\in\NN$. Next, we assign four children to $\root_2$: $\Chi(\root_2)=\{s_2^{(1)}, s_2^{(2)}, s_2^{(3)}, \root_3\}$. Each vertex $s_2^{(1)}, s_2^{(2)}, s_2^{(3)}$ generates its own stationary branch, while $\root_3$ is assigned a total of $2^3 = 8$ children. By continuing this process \emph{ad infinitum}, we obtain a directed tree in which no vertex is fertile (see the representation in Figure \ref{fig:ex-no-fertile}). Therefore, the conclusion follows from Theorem \ref{thm:fertile}.
\end{example}
\begin{figure}[H]
    \centering
    \includegraphics[width=0.8\linewidth]{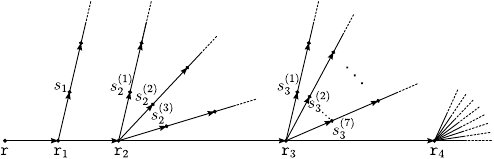}
    \caption{Tree with no fertile vertex}
    \label{fig:ex-no-fertile}
\end{figure}

One last comment on mixing shifts with no hypercyclic algebra. In \cite[Example 4.8]{Karl1}, the authors define a big class of mixing weighted backward shifts $B_{\lambda, q}$ acting $\ell^2(A)$, where $A$ is a rooted tree with finitely many children for each vertex and $q>1$. One can verify that $B_{\lambda,q}$ has no hypercyclic algebra as soon as the tree $A$ satisfies
\[\limsup_{n\to+\infty}\sup_{u\in\Chi^n(\root)}\frac{n^{q-1}}{\prod_{k=1}^n|\Chi(\Par^k(u))|}<+\infty\]
(in particular, for $q=2$ the shift $B_{\lambda,2}$ is called the \emph{Bergman shift}). This example constitute a large class of mixing shifts with no hypercyclic algebra.

\section{Concluding remarks and open questions}\label{sec:conc}

In this paper, we have explored the existence of hypercyclic algebras for backward shifts on sequence spaces of directed trees, equipped with the coordinatewise product. Our investigation has led to several key findings, including characterizations for shifts on rooted trees. The study of convolution products is left open for exploration and can lead to interesting results on $\ell^1$-spaces of a tree. Whereas the coordinatewise product is naturally defined on sequence spaces of any directed tree, there are multiple ways of generalizing the convolution product depending on the geometry of the tree in question. Thus, the mere consideration of a product with interesting applications is a compelling question.
\begin{question}
    What definition of convolution product in the sequence space of a directed tree would provide good applications concerning hypercyclic algebras?
\end{question}
An example of good application would be to obtain the equivalence between hypercyclicity and the existence of a hypercyclic algebra, since this is something that happens for weighted shifts on $\ell^1(\NN)$ or $\omega$ with the convolution product (see \cite[Theorem 1.3]{BCP} and \cite[Corollary 3.9]{karl-falco}). It would be also interesting to have a negative result here: to find a generalization of the convolution product such that this equivalence is not true anymore. Of course, a correct generalization would give the classical product when the tree is $A=\NN$.

In the unrooted case, however, the study of coordinatewise product is not fully completed. Is it not clear if the necessary conditions in Theorem \ref{thm:necessary:unrooted} are also sufficient. Thus, we are left with the following question.
\begin{question}
    Does (v)$\Rightarrow$(i) or (ii) in Theorem \ref{thm:necessary:unrooted}?
\end{question}
This result is our current best candidate for a characterization, although the technique present limitations when the condition $\lambda(\Par^{n_k}(v)\to v)\xrightarrow{k\to\infty}0$ doesn't happen. When it comes to the existence of hypercyclic algebras for certain classes of operators, on the other hand, 
Corollary \ref{corol:rolewicz} constitutes a satisfying partial answer for Rolewicz operators. Also, Corollary \ref{corol:symmetric:carac} is a pleasing characterization in the symmetric case.

Similar questions arise naturally for $c_0$-spaces of an unrooted tree. For instance, it is not clear if one can use Theorem \ref{thm:necessary:c0} to get an example of a weight $\lambda$ and an unrooted tree $A$ such that $B_\lambda$ is hypercyclic on $c_0(A)$ but has no hypercyclic algebra. Of course it could be the case that $B_\lambda$ has a hypercyclic algebra as soon as it is hypercyclic on $c_0(A)$, but it could also happen that Theorem \ref{hcalg:c0:unrooted} characterizes weighted shifts supporting hypercyclic algebras on $c_0(A)$.

\begin{question}
    Does condition (iii) in Theorem \ref{thm:necessary:c0} imply the existence of a hypercyclic algebra? Are the sufficient conditions in Theorem \ref{hcalg:c0:unrooted} also necessary for the existence of a hypercyclic algebra?
\end{question}

Still on Theorem \ref{thm:necessary:unrooted}, the current statement of (ii) or (iii) relates to a deeper question on the theory of hypercyclic algebras: the presence of dense sets of generators. It is a classical manipulation technique to write the set of hypercyclic vectors for an operator as the countable intersection of dense open sets, which implies that it is either empty or residual. Is it not clear if something similar could be done for hypercyclic algebras.

\begin{question}
    Let $T:X\to X$ be a continuous linear map acting on an $F$-space $X$. Suppose that $T$ supports a hypercyclic algebra. Is it true that the set of elements $x\in X$ generating a hypercyclic algebra for $X$ is residual? Is it at least dense?
\end{question}

Concerning the existence of mixing shifts with no hypercyclic algebras, we have solved the problem for the rooted case on $\ell^p$-spaces and $c_0$-spaces: on $\ell^1$ and $c_0$-spaces we know that hypercyclicity implies the existence of a hypercyclic algebra, whereas on $\ell^p$-spaces with $1<p<+\infty$, the tree must have a fertile vertex. However, the question remains open in the unrooted case.

\begin{question} 
    Can we characterize the unrooted directed trees $A$ supporting a mixing backward shift $B_\lambda:X\to X$ with no hypercyclic algebra, with $X=\ell^p(A), 1\leq p<+\infty$, or $X=c_0(A)$?
\end{question}
For $\ell^p$-spaces, with $1<p<\infty$, one can adapt the proof of Theorem \ref{thm:fertile} in order to construct a mixing backward shift with no hypercyclic algebra as soon as the tree has a fertile vertex. The inverse implication, as well as the case of $\ell^1$ and $c_0$-spaces are open.

A last open problem would be to extend the study developed in this article to locally convex spaces, in the spirit of \cite[Section 10]{Karl2}. Indeed, the theory of hypercyclic algebras has been developed in the context of Fréchet spaces (and even $F$-spaces), thus is it natural to expand our results in this direction. We believe that our techniques can provide complete characterization in the case of rooted trees. It would be interesting to study the gap that we mentioned in our study between the rooted and the unrooted cases.

\section*{Acknowledgments}
This work was partially supported by the INSMI PEPS JCJC research grant. The second author was also supported by CNPq grant 406457/2023-9.

\bibliographystyle{abbrv}
\bibliography{hc-alg-trees.bib}

\end{document}